\PassOptionsToPackage{unicode}{hyperref}
\PassOptionsToPackage{naturalnames}{hyperref}
\documentclass[a4paper]{amsart}
\usepackage{latexsym,bm,stmaryrd}
\usepackage[a4paper,hmargin=26mm,vmargin=28mm]{geometry}
\usepackage{latexsym,bm}
\usepackage[svgnames]{xcolor}
\usepackage{xparse}
\usepackage{listings}
\usepackage{tikz}
\usetikzlibrary{matrix}
\usepackage[utf8]{inputenc}
\usepackage[english]{babel}
\usepackage[T1]{fontenc}
\usepackage{amsmath, amsfonts, amssymb, amsthm, mathrsfs, pb-diagram}
\usepackage{lmodern}
\usepackage{fullpage}
\usepackage{microtype}
\usepackage{enumerate}
\usepackage{mathtools}
\usepackage{enumitem}
\setlist[enumerate]{label=\alph*\upshape), nolistsep}

\usepackage{todonotes}

\NewDocumentCommand\Cmd{ sm }{\textsf{\textbackslash #2}\IfBooleanT{#1}{$\{\ldots\}$}}

\usepackage{etoolbox}
\usepackage{aliascnt}

\usepackage[pagebackref=false]{hyperref}

\newcommand\enumref[2]{\hyperref[#2]{\autoref*{#1}(\autoref*{#2})}}

\usepackage[misc]{ifsym}

\def\NewTheorem#1{%
  \newaliascnt{#1}{equation}%
  \newtheorem{#1}[#1]{#1}%
  \aliascntresetthe{#1}%
  \expandafter\def\csname #1autorefname\endcsname{#1}%
}

\newcommand\blam{{\boldsymbol\lambda}}

\newcommand\bmu{{\boldsymbol\mu}}

\DeclareMathOperator\Shape{Shape}

\let\<=\langle
\let\>=\rangle
\def\({\big(}
\def\){\big)}

\def\Z{\mathbb{Z}}
\def\Q{\mathbb{Q}}

\def\a{\mathfrak{a}}

\def\t{\mathfrak{t}}
\def\s{\mathfrak{s}}
\def\u{\mathfrak{u}}
\def\v{\mathfrak{v}}

\def\lam{\lambda}

\def\Sym{\mathfrak{S}}

\newcommand\HH{\mathscr{H}}

\def\P{\mathscr{P}}

\def\bQ{\mathbf{Q}}
\def\bu{\mathbf{u}}
\def\mL{\mathcal{L}}
\def\fm{\mathfrak{m}}
\def\fn{\mathfrak{n}}

\def\rmf{{\rm f}}
\def\rmg{{\rm g}}
\def\rmF{{\rm F}}
\def\rmm{{\rm m}}
\def\rmn{{\rm n}}

\DeclareMathOperator\res{res}

\DeclareMathOperator\Std{Std}

\def\wres{c}

\def\wgamma{r}

\def\mff{\mathfrak{f}}
\def\mfg{\mathfrak{g}}

\begin{document}

\title{On the seminormal bases and dual seminormal bases of the cyclotomic Hecke algebras of type $G(\ell,1,n)$}
\author{Jun Hu}\address{MIIT Key Laboratory of Mathematical Theory and Computation in Information Security\\
  School of Mathematical and Statistics\\
  Beijing Institute of Technology\\
  Beijing, 100081, P.R. China}
\email{junhu404@bit.edu.cn}

\author{Shixuan Wang\textsuperscript{\Letter}}\thanks{\Letter Corresponding author: Shixuan Wang \qquad Email: 3120195740@bit.edu.cn}\address{School of Mathematical and Statistics\\
  Beijing Institute of Technology\\
  Beijing, 100081, P.R. China}
\email{3120195740@bit.edu.cn}

\numberwithin{equation}{section}
\newtheorem{prop}[equation]{Proposition}
\newtheorem{thm}[equation]{Theorem}
\newtheorem{cor}[equation]{Corollary}
\newtheorem{conj}[equation]{Conjecture}
\newtheorem{hcond}[equation]{Homogeneous Admissible Condition}
\newtheorem{lem}[equation]{Lemma}
\newtheorem{problem}[equation]{Conjecture}
\newtheorem{mainthm}[equation]{Main Theorem}
\theoremstyle{definition}
\newtheorem{dfn}[equation]{Definition}
\theoremstyle{remark}
\newtheorem{rem}[equation]{Remark}
\newtheorem{dnt}[equation]{Definition and Theorem}
\newenvironment{Def}[2][Definition~\theequation.]{\refstepcounter{equation}\par\medskip
	\noindent \textbf{Definition~ \theequation.} \hskip \labelsep {\bfseries #2.} \rmfamily}{\medskip}

\begin{abstract} This paper studies the seminormal bases $\{\mff_{\s\t}\}, \{{\rm f}_{\s\t}\}$, and the dual seminormal bases $\{\mfg_{\s\t}\}, \{{\rm g}_{\s\t}\}$ of the non-degenerate and the degenerate cyclotomic Hecke algebras $\HH_{\ell,n}$ of type $G(\ell,1,n)$. We present some explicit formulae for the constants $\alpha_{\s\t}:=\mfg_{\s\t}/\mff_{\s\t}\in K^\times$, $a_{\s\t}:={\rm g}_{\s\t}/{\rm f}_{\s\t}\in K^\times$ in terms of the $\gamma$-coefficients $\{\gamma_\u, \gamma'_\u\}$ and the $r$-coefficients $\{r_\u, r'_\u\}$ of $\HH_{\ell,n}$. In particular, we answer a question \cite[Remark 3.6]{Ma} of Mathas on the rationality of square roots of some quotients of products of $\gamma$-coefficients. We obtain some explicit formulae for the expansion of each seminormal bases of $\HH_{\ell,n-1}$ as a linear combination of the seminormal bases of $\HH_{\ell,n}$ under the natural inclusion $\HH_{\ell,n-1}\hookrightarrow\HH_{\ell,n}$.
\end{abstract}

\maketitle
\setcounter{tocdepth}{1}

\section{Introduction}

Let $\ell,n$ be two positive integers. The cyclotomic Hecke algebras $\HH_{\ell,n}$ of type $G(\ell,1,n)$, also known as Ariki-Koike algebras, can be viewed as some generalizations of the Iwahori-Hecke algebras of types $A$ and $B$. They were introduced by Brou\'e and Malle (\cite{BM:cyc}), and independently by Ariki and Koike (\cite{A2},\cite{AK}), and they play an important role in the modular representation theory of finite groups of Lie type over fields of non-defining characteristic. These algebras have been studied extensively both because of their rich representation theory and because of their close relationships with the affine Hecke algebras of type $A$, KLR algebras, BGG category $\mathcal{O}$ and geometric representation theory, see \cite{A3}, \cite{BK}, \cite{BK:GradedKL}, \cite{BK:Gradedecomp}, \cite{Ch}, \cite{Gro} and \cite{Klesh:book}.

The cyclotomic Hecke algebras $\HH_{\ell,n}$ is cellular in the sense of Graham and Lehrer (\cite{GL}). Using the cellular bases of $\HH_{\ell,n}$ constructed in \cite{DJM:cyc} and \cite{Ma}, Mathas has constructed in \cite{Ma} a seminormal basis $\{\mff_{\s\t}\}$ and a dual seminormal basis $\{\mfg_{\s\t}\}$ for the non-degenerate cyclotomic Hecke algebras $\HH_{\ell,n}(q,\bQ)$ when they are semisimple and $q\neq 1$. These seminormal bases are important not only for the semisimple representation theory of $\HH_{\ell,n}$, but also for the modular representation theory of $\HH_{\ell,n}$, see \cite{Ma2}, \cite{HuMathas:GradedCellular} and \cite{HuMathas:SeminormalQuiver}. For the degenerate cyclotomic Hecke algebra $H_{\ell,n}(\bu)$, there are similar constructions and results (e.g., seminormal basis $\{{\rm f}_{\s\t}\}$, dual seminormal basis $\{{\rm g}_{\s\t}\}$) in \cite{AMR} and \cite{Zh}. Much of the theory on the seminormal bases applied in the paper goes back to Murphy in the symmetric groups and associated Hecke algebras cases (i.e., $\ell=1$), see \cite{Mur1}, \cite{Mur2}, \cite{Mur3} and \cite{Mur4}. By the semisimplicity criterion of $\HH_{\ell,n}$ and some eigenvalue (w.r.t. the Jucys-Murphy operators of $\HH_{\ell,n}$) consideration, we see that $\alpha_{\s\t}:=\mfg_{\s\t}/\mff_{\s\t}\in K^\times$ and  $a_{\s\t}:={\rm g}_{\s\t}/{\rm f}_{\s\t}\in K^\times$, where $K$ is the ground field. However, these constants $\alpha_{\s\t}, a_{\s\t}$ were not explicitly known as rational functions in the literature at the moment. The purpose of this paper is to give some explicit formulae of $\alpha_{\s\t}, a_{\s\t}$ for both the non-degenerate and the degenerate cyclotomic Hecke algebras. To state our main result, we need some definitions and notations.

Let $R$ be an integral domain and $q\in R^\times$. Let $\bQ=(Q_1,\cdots,Q_\ell)$, where $Q_1,\cdots,Q_\ell \in R$. The non-degenerate cyclotomic algebra $\HH_{\ell,n}(q,\bQ)$ of type $G(\ell,1,n)$ is the unital associative $R$-algebra with generators $T_0,T_1,\cdots,T_{n-1}$ and the following defining relations:
$$\begin{aligned}
&(T_0-Q_1)\cdots(T_0-Q_\ell)=0;\\
&T_0T_1T_0T_1=T_1T_0T_1T_0;\\
&(T_i-q)(T_i+1)=0,\quad \forall\,1\leq i\leq n-1;\\
&T_iT_j=T_jT_i,\,\,\forall\,1\leq i<j-1<n-1,\\
&T_iT_{i+1}T_i=T_{i+1}T_iT_{i+1},\,\,\forall\,1\leq i<n-1.
\end{aligned}$$
Following \cite[\S2]{DJM}, we define $$ \mL_m:=q^{1-m}T_{m-1}\cdots T_1T_0T_1\cdots T_{m-1},\,\, m=1,2,\cdots,n, $$
and call them the Jucys-Murphy operators of $\HH_{\ell,n}(q,\bQ)$.

Let $\P_n$ be the set of multipartitions of $n$. For each $\blam\in\P_n$, let $\Std(\blam)$ be the set of standard $\blam$-tableaux. Let $\{\fm_{\s\t}|\s,\t\in\Std(\blam),\blam\in\P_n\}$ be the Dipper-James-Mathas cellular basis of $\HH_{\ell,n}(q,\bQ)$. The definition of $\fm_{\s\t}$ makes use of a ``trivial representation'' of the Hecke algebra $\HH_q(\Sym_\blam)$ associated to a standard Young subgroup $\Sym_\blam$ of $\Sym_n$.
Suppose that $q\neq 1$ and $\HH_{\ell,n}(q,\bQ)$ is semisimple. Let $\{\mff_{\s\t}|\s,\t\in\Std(\blam),\blam\in\P_n\}$ be the seminormal basis of $\HH_{\ell,n}(q,\bQ)$ corresponding to the cellular basis $\{\fm_{\s\t}|\s,\t\in\Std(\blam),\blam\in\P_n\}$. Replacing the ``trivial representation'' of $\HH_q(\Sym_\blam)$ in the construction of $\fm_{\s\t}$ by the ``sign representation'' of $\HH_q(\Sym_\blam)$, one can also get a second cellular basis $\{\fn_{\s\t}\}$, called the dual cellular basis, of $\HH_{\ell,n}(q,\bQ)$. Starting from the dual cellular basis $\{\fn_{\s\t}|\s,\t\in\Std(\blam),\blam\in\P_n\}$ of $\HH_{\ell,n}(q,\bQ)$, we can construct a dual seminormal basis $\{\mfg_{\s\t}|\s,\t\in\Std(\blam),\blam\in\P_n\}$.  We refer the readers to Section 2 and Section 3 for unexplained notations here. The following theorem is the first main result of this paper.

\begin{thm}\label{mainthm1}
Let	$\blam\in\P_n$ and $\s,\t\in\Std(\blam)$. Suppose $q\neq 1$ and $\HH_{\ell,n}(q,\bQ)$ is semisimple. Then
	$$\alpha_{\s\t}:=\mfg_{\s\t}/\mff_{\s\t}=(-q)^{-\ell(d(\s'))-\ell(d(\t'))}\frac{\gamma_{\t_{\blam}}\gamma_{\t^{\blam'}}'}{\gamma_{\s}\gamma_{\t}}=
(-q)^{\ell(d(\s'))+\ell(d(\t'))}\frac{\gamma_{\s'}'\gamma_{\t'}'}{\gamma_{\t_{\blam}}\gamma_{\t^{\blam'}}'}, $$
where for each $\u\in\Std(\blam)$, $\gamma_\u$ is the $\gamma$-coefficient defined in Definition \ref{gamma1}, $\gamma'_\u$ is defined as in Definition \ref{prime}. \end{thm}

There is a natural algebra embedding $\iota: \HH_{\ell,n-1}(q,\bQ)\hookrightarrow\HH_{\ell,n}(q,\bQ)$, which is defined on generators by $\iota(T_i):=T_i$ for $0\leq i<n-1$. In order to avoid the confusion between the notations for $\HH_{\ell,n-1}(q,\bQ)$ and $\HH_{\ell,n}(q,\bQ)$. We add a superscript $(n)$ to indicate that it is the notation for $\HH_{\ell,n}(q,\bQ)$. Let $\bmu\in\P_{n-1}$ and $\s,\t\in\Std(\blam)$. Under the embedding $\iota$, we have \begin{equation}\label{beta}
\mff_{\s\t}^{(n-1)}=\sum_{\blam\in\P_n}\sum_{\u,\v\in\Std(\blam)}\beta_{\u\v}^{\s\t}\mff_{\u\v}^{(n)},
\end{equation}
where $\beta_{\u\v}^{\s\t}\in K$ for each pair $(\u,\v)$. The following theorem is the second main result of this paper.

\begin{thm}\label{mainthm2}
	Let $\bmu\in\P_{n-1}, \blam\in\P_n$, and $\s,\t\in\Std(\bmu)$, $\u,\v\in\Std(\blam)$. Suppose $q\neq 1$ and $\HH_{\ell,n}(q,\bQ)$ is semisimple. Then $\beta_{\u\v}^{\s\t}\neq 0$ if and only if $\u\downarrow_{n-1}=\s$ and $\v\downarrow_{n-1}=\t$. In that case we have $$\beta_{\u\v}^{\s\t}=\frac{\gamma_{\s}^{(n-1)}}{\gamma_{\u}^{(n)}}=\frac{\gamma_{\t}^{(n-1)}}{\gamma_{\v}^{(n)}}.$$
\end{thm}
If we set $\ell:=1$ and $Q_1:=1$, then the above two main results Theorems \ref{mainthm1} and \ref{mainthm2} give the corresponding results for the classical semisimple Iwahori-Hecke algebra $\HH_q(\Sym_n)$ associated to the symmetric group $\Sym_n$ with Hecke parameter $q\neq 1$.

The degenerate case is parallel to the non-degenerate case with slight modification. Let $\bu=(u_1,\cdots,u_\ell)$, where $u_1,\cdots,u_\ell \in R$. The degenerate cyclotomic Hecke algebra $H_{\ell,n}(\bu)$ of type $G(\ell,1,n)$ is the unital associative $R$-algebra with generators $s_{1},\cdots,s_{n-1},L_1,\cdots,L_n$ and the following defining relations:
$$\begin{aligned}
	&(L_1-u_1)\cdots(L_1-u_\ell)=0;\\
	&s_{i}^{2}=1,\quad \forall\,1\leq i\leq n-1;\\
	&s_{i}s_{j}=s_{j}s_{i},\,\,\forall\,1\leq i<j-1<n-1,\\
	&s_is_{i+1}s_i=s_{i+1}s_is_{i+1},\,\,\forall\,1\leq i<n-1,\\
&L_iL_k=L_kL_i,\,\, s_iL_l=L_ls_i,\,\,1\leq i<n, 1\leq k,l\leq n, l\neq i,i+1,\\
&L_{i+1}=s_iL_is_i+s_i,\,\,\, 1\leq i<n .
\end{aligned}
$$
The elements $L_1,\cdots,L_n$ are called the Jucys-Murphy elements of the degenerate cyclotomic Hecke algebra $H_{\ell,n}(\bu)$.

As in the non-degenerate case, we have a cellular basis $\{{\rm m}_{\s\t}|\s,\t\in\Std(\blam),\blam\in\P_n\}$ as well as a dual cellular basis $\{{\rm n}_{\s\t}|\s,\t\in\Std(\blam),\blam\in\P_n\}$ of $H_{\ell,n}(\bu)$. Suppose that $H_{\ell,n}(\bu)$ is semisimple. Let $\{{\rm f}_{\s\t}|\s,\t\in\Std(\blam),\blam\in\P_n\}$ be the seminormal basis of $H_{\ell,n}(\bu)$ corresponding to the cellular basis $\{{\rm m}_{\s\t}|\s,\t\in\Std(\blam),\blam\in\P_n\}$. Let $\{{\rm g}_{\s\t}|\s,\t\in\Std(\blam),\blam\in\P_n\}$ be the dual seminormal basis of $H_{\ell,n}(\bu)$ corresponding to the dual cellular basis $\{{\rm n}_{\s\t}|\s,\t\in\Std(\blam),\blam\in\P_n\}$. Then we have that $a_{\s\t}:={\rm g}_{\s\t}/{\rm f}_{\s\t}\in K^\times$ for any $\s,\t\in\Std(\blam), \blam\in\P_n$. The following two theorems are the analogues of Theorems \ref{mainthm1}, \ref{mainthm2} for the degenerate case.

\begin{thm}\label{mainthm1b}
	Let $\blam\in\P_n$ and $\s,\t\in\Std(\blam)$. Suppose $H_{\ell,n}(\bu)$ is semisimple. Then
	$$ a_{\s\t}:={\rm g}_{\s\t}/{\rm f}_{\s\t}=(-1)^{-\ell(d(\s'))-\ell(d(\t'))}\frac{r_{\t_{\blam}}r_{\t^{\blam'}}'}{r_{\s}r_{\t}}=
(-1)^{\ell(d(\s'))+\ell(d(\t'))}\frac{r_{\s'}'r_{\t'}'}{r_{\t_{\blam}}r_{\t^{\blam'}}'}, $$
where for each $\u\in\Std(\blam)$, $r_\u$ is the $r$-coefficient defined in Definition \ref{gamma2}, $r'_\u$ is defined as in Definition \ref{prime2}. \end{thm}

As in the non-degenerate case, we also have an algebra embedding
$\iota_1: H_{\ell,n-1}(\bu)\hookrightarrow H_{\ell,n}(\bu)$, which is defined on generators by $\iota_1(s_i):=s_i$, $\iota_1(L_k)=L_k$, for $1\leq i<n-1, 1\leq k\leq n-1$. In order to avoid the confusion between the notations for $H_{\ell,n-1}(\bu)$ and $H_{\ell,n}(\bu)$. We add a superscript $(n)$ to indicate that it is the notation for $H_{\ell,n}(\bu)$. Let $\bmu\in\P_{n-1}$ and $\s,\t\in\Std(\blam)$. Under the embedding $\iota_1$, we have \begin{equation}\label{betab}
{\rm f}_{\s\t}^{(n-1)}=\sum_{\blam\in\P_n}\sum_{\u,\v\in\Std(\blam)}b_{\u\v}^{\s\t}{\rm f}_{\u\v}^{(n)},
\end{equation}
where $b_{\u\v}^{\s\t}\in K$ for each pair $(\u,\v)$.

\begin{thm}\label{mainthm2b}
	Let $\bmu\in\P_{n-1}, \blam\in\P_n$, and $\s,\t\in\Std(\bmu)$, $\u,\v\in\Std(\blam)$. Suppose $H_{\ell,n}(\bu)$ is semisimple. Then $b_{\u\v}^{\s\t}\neq 0$ if and only if $\u\downarrow_{n-1}=\s$ and $\v\downarrow_{n-1}=\t$. In that case we have $$b_{\u\v}^{\s\t}=\frac{r_{\s}^{(n-1)}}{r_{\u}^{(n)}}=\frac{r_{\t}^{(n-1)}}{r_{\v}^{(n)}}.$$
\end{thm}
If we set $\ell:=1$ and $u_1:=0$, then the above two main results Theorems \ref{mainthm1b} and \ref{mainthm2b} give the corresponding results for the semisimple symmetric group algebra $K[\Sym_n]$.\medskip

The content of the paper is organised as follows. In Section 2 we give some preliminary results on the structure and representation theory of the cyclotomic Hecke algebras $\HH_{\ell,n}$ of type $G(\ell,1,n)$. In particular, we shall recall the construction of cellular bases and seminormal bases of $\HH_{\ell,n}$. In Section 3 we first recall the construction of the dual cellular bases and the dual seminormal bases of $\HH_{\ell,n}(q,\bQ)$. Then we reveal some hidden relationship between various $\gamma$-coefficients in Lemma \ref{gam'tgamt'}. Combining this with the use of certain remarkable invertible elements $\Phi_s$ introduced in Mathas's work \cite{Ma}, we finally
give the proof of the main results Theorem \ref{mainthm1} and Theorem \ref{mainthm2}. In Section 4 we deal with the degenerate cyclotomic Hecke algebra $H_{\ell,n}(\bu)$. The argument is similar as the non-degenerate case. In particular, we give the proof of the main results Theorem \ref{mainthm1b} and Theorem \ref{mainthm2b}.

\bigskip\bigskip
\centerline{Acknowledgements}
\bigskip

The research was support by the National Natural Science Foundation of China (No. 12171029).
\bigskip

\section{Preliminary}

Let $\HH_{\ell,n}\in\{\HH_{\ell,n}(q,\bQ),H_{\ell,n}(\bu)\}$. Let $\Sym_n$ be the symmetric group on $\{1,2,\cdots,n\}$. For each $1\leq i<n$, we set $s_i:=(i,i+1)$. A word $w=s_{i_{1}}s_{i_{2}}\ldots s_{i_{k}}$ for $w\in \Sym_{n}$ is called a reduced expression of $w$ if $k$ is minimal; in this case we say $w$ has length $k$ and we write $\ell(w)=k$. Given a reduced expression $s_{i_{1}}\cdots s_{i_{k}}$ of $w\in \Sym_{n}$, we define $T_{w}=T_{i_{1}}\cdots T_{i_{k}}$, which is independent of the choice of the reduced expression of $w$ because the braid relations hold in $\HH_{\ell,n}(q,\bQ)$. Let ``$\ast$'' be the unique anti-involution of $\HH_{\ell,n}$ which fixes its defining generators.

Dipper, James and Mathas have shown in \cite{DJM:cyc} that the algebra $\HH_{\ell,n}$ is cellular in the sense of \cite{GL}. To recall the cellular structure given in  \cite{DJM:cyc}, we need some combinatorial notions and notations.
Let $a$ be a positive integer. A partition of $a$ is a weakly decreasing sequence $\lam=(\lam_{1},\lam_{2},\cdots)$ of non-negative integers such that $|\lam|:=\Sigma_{i\geq 1}\lam_{i}=a$.
Let $\lam=(\lam_{1},\lam_{2},\ldots)\vdash a$ be a partition of $a$. We define $\lam'=(\lam'_{1},\lam'_{2},\ldots)$, where for each $i$, $\lam'_{i}:=\#\{j|\lam_{j} \geq i\}$. Then $\lam'$ is again a partition of $a$ and is called the conjugate of $\lam$. A multipartition of $n$ is an $\ell$-tuple $\blam=(\lam^{(1)},\cdots,\lam^{(\ell)})$ of partitions such that $|\lam^{(1)}|+\cdots+|\lam^{(\ell)}|=n$. We define the Young diagram of $\blam$ to be $[\blam]:=\{(i,j,c)|1\leq j\leq \lam_{i}^{(c)}, 1\leq c\leq \ell\}$. A $\blam$-tableau $\t$ is a bijective map $\t: [\blam]\rightarrow\{1,2,\cdots,n\}$. If the $\blam$-tableau $\t$ satisfies that
$\t(i,j,l)\leq\t(a,b,l)$ for any $i\leq a$ and $j\leq b$ and $1\leq l\leq\ell$, then we say $\t$ is standard. We use $\Std(\blam)$ to denote the set of standard $\blam$-tableaux. If $\t\in\Std(\blam)$, then we set $\Shape(\t):=\blam$, and we can write $\t=(\t^{(1)},\cdots,\t^{(\ell)})$, where each $\t^{(i)}$ is a standard $\lam^{(i)}$-tableaux.

Let $\P_n$ be the set of multipartitions of $n$. For each multipartition $\blam=(\lam^{(1)},\cdots,\lam^{(\ell)})\in\P_n$, let $\Sym_\blam$ be the corresponding standard Young subgroup of $\Sym_n$. That is, $$\begin{aligned}
\Sym_\blam:&=\Sym_{\{1,\cdots,\lam_1^{(1)}\}}\times\Sym_{\{\lam_1^{(1)}+1,\cdots,\lam_2^{(1)}\}}\times\cdots\times\Sym_{\{|\lam^{(1)}|-\lam_{b_1}+1,\cdots,|\lam^{(1)}|\}}\times\cdots\\
&\qquad\times\Sym_{\{n-|\lam^{(\ell)}|+1,\cdots,n-|\lam^{(\ell)}|+\lam_{1}^{(\ell)}\}}\times\cdots\times \Sym_{\{n-|\lam^{(\ell)}_{b_\ell}|+1,\cdots,n\}},
\end{aligned}$$
where $b_i:=(\lam^{(i)'})_1$ for $i=1,2,\cdots,\ell$. For each $\blam\in\P_n$, we define $$
\blam':=\bigl(\lam^{(\ell)'},\cdots,\lam^{(1)'}\bigr) ,
$$
and call it the conjugate of $\blam$. For each $\t\in\Std(\blam)$, we define $$
\t'=\bigl(\t^{(\ell)'},\cdots,\t^{(1)'}\bigr) .
$$
Then $\t'\in\Std(\blam')$.

Let $\t^\blam$ be the initial standard $\blam$-tableaux in which the numbers $1,2,\cdots,n$ are entered in order first along the rows of $\t^{\lam^{(1)}}$ and then the rows of $\t^{\lam^{(2)}}$ and so on. We define $\t_{\blam}:=(\t^{\blam'})'$. In particular, $\t_\blam$ is the standard $\blam$-tableaux in which the numbers $1,2,\cdots,n$ are entered in order first along the columns of $\t^{\lam^{(\ell)}}$ and then the columns of $\t^{\lam^{(\ell-1)}}$ and so on. For each $\t\in\Std(\blam)$, let $d(\t)\in\Sym_n$ be the unique element in $\Sym_n$ such that $\t^\blam d(\t)=\t$, and we set $w_\blam:=d(\t_\blam)$.

For any $\blam,\bmu\in\P_n$, we write $\blam\unrhd\bmu$ if for all $1\leq s\leq\ell$ and all $i\geq 1$,
$$\sum_{t=1}^{s-1}|\lam^{(t)}|+\sum_{j=1}^{i}\lam^{(s)}_{j}\geq \sum_{t=1}^{s-1}|\mu^{(t)}|+\sum_{j=1}^{i}\mu^{(s)}_{j}.$$
Clearly $\P_n$ is a poset with respect to the partial order ``$\unrhd$''.

If $\blam\unrhd\bmu$ and $\blam\neq\bmu$, then we write $\blam\rhd\bmu$. Let $\s\in\Std(\blam), \t\in\Std(\bmu)$. We write $\s\unrhd\t$ if for any $1\leq k\leq n$, ${\rm{Shape}}(\s\!\downarrow_{\{1,2,\cdots,k\}})\unrhd{\rm{Shape}}(\t\!\downarrow_{\{1,2,\cdots,k\}})$. If $\s\unrhd\t$ and $\s\neq\t$ then we write $\s\rhd\t$.
Clearly, $\t^\blam\unrhd\s\unrhd\t_\blam$ for any $\s\in\Std(\blam)$.

\begin{dfn}\text{(\cite{Ma}, \cite{AMR})}\label{cellularbases1} Let $\blam\in\P_n$ and $\s,\t\in\Std(\blam)$. We define $$\begin{aligned}
\fm_{\s\t}:&=T_{d(\s)}^*\Bigl(\sum_{w\in\Sym_\blam}T_w\Bigr)\Bigl(\prod_{s=2}^{\ell}\prod_{k=1}^{|\lam^{(1)}|+\cdots+|\lam^{(s-1)}|}(\mL_k-Q_s)\Bigr)T_{d(\t)},\\
{\rm m}_{\s\t}:&=d(\s)^{-1}\Bigl(\sum_{w\in\Sym_\blam}w\Bigr)\Bigl(\prod_{s=2}^{\ell}\prod_{k=1}^{|\lam^{(1)}|+\cdots+|\lam^{(s-1)}|}(L_k-u_s)\Bigr)d(\t),
\end{aligned}
$$
\end{dfn}

\begin{thm}\text{(\cite{DJM}, \cite{AMR}, \cite{Zh})} With respect to the poset $(\P_n, \unrhd)$ and the anti-involution $\ast$, the set $\{\fm_{\s\t}|\s,\t\in\Std(\blam),\blam\in\P_n\}$ forms a cellular basis of $\HH_{\ell,n}(q,\bQ)$, while the set $\{{\rm m}_{\s\t}|\s,\t\in\Std(\blam),\blam\in\P_n\}$ forms a cellular basis of $H_{\ell,n}(\bu)$.
\end{thm}

One of the remarkable properties of the basis $\{\fm_{\s\t}\}$ is that it can be defined over an arbitrary ground ring, though the computation of the product of these bases can be rather complicated. When the Hecke algebra $\HH_{\ell,n}(q,\bQ)$ is semisimple, there is another basis (called seminormal basis) of $\HH_{\ell,n}(q,\bQ)$ which is much easier for calculation. Henceforth we assume that $q\neq 1$. Let us recall the following criteria of semisimplicity for $\HH_{\ell,n}(q,\bQ)$.

\begin{lem}\text{(\cite{A1})} Let $R=K$ be a field. Suppose $1\neq q\in K^\times$. Then $\HH_{\ell,n}(q,\bQ)$ is semisimple if and only if \begin{equation}\label{ss0}
\prod_{i=1}^n(1+q+q^2+\cdots+q^{i-1})\prod_{\substack{1\leq i<j\leq\ell\\ |d|<n}}(q^dQ_i-Q_j)\in K^\times .
\end{equation}
\end{lem}

For any $\t=(\t^{(1)},\cdots,\t^{(\ell)})\in\Std(\blam)$ and any $1\leq k\leq n$, we define \begin{equation}
\res_{\t}(k)=q^{j-i}Q_c,\quad  \text{if $k$ appears in row $i$ and column $j$ of $\t^{(c)}$}
\end{equation}
We also define $R(k):=\{\res_\t(k) | \t\in\Std(\blam), \blam\vdash n\}$.

The condition (\ref{ss0}) is actually equivalent to the following statement: \begin{equation}\label{separacontent1}
\begin{matrix}\text{for any $\blam,\bmu\in\P_n$, $\s\in\Std(\blam), \t\in\Std(\bmu)$, if $\s\neq\t$, then there exists}\\
\text{ $1\leq k\leq n$ such that $\res_\s(k)-\res_\t(k)\in K^\times$.}\end{matrix}
\end{equation}

\begin{dfn}\label{Ft}(\cite{Mur3}, \cite[Definition 2.4]{Ma}) Suppose $q\neq 1$ and (\ref{ss0}) holds. Let $\blam\in\P_n$ and $\t\in\Std(\blam)$. We define
$$F_{\t}=\prod\limits^n\limits_{k=1}\prod\limits_{\substack{c\in R(k)\\c\neq \res_{\t}(k)}}\frac{\mL_k-c}{\res_{\t}(k)-c}.$$
\end{dfn}

For any $\blam\in\P_n$ and $\s,\t\in\Std(\blam)$, we define \begin{equation}\label{seminormal}
\mff_{\s\t}^{(n)}:=F_\s \fm_{\s\t}F_\t .\end{equation}
When the context is clear, we shall omit the superscript ``$(n)$'' and write $\mff_{\s\t}$ instead of $\mff_{\s\t}^{(n)}$.

For any $k\in\Z^{\geq 0}$, we define $[k]_q=\sum\limits_{i=0}^{k-1}q^i$. For any $m\in\Z^{\geq 0}$, we set $[m]^!_q=[1]_q[2]_q\cdots[m]_q$. If $\blam=(\lam^{(1)},\cdots,\lam^{(\ell)})\in\P_n$, then we define $[\blam]^!_q=\prod\limits_{c=1}^\ell\prod\limits_{i\geq1}[\lambda^{(c)}_i]^!_q$.

\begin{dfn}{\rm (\cite{Mur3}, \cite[(3.17)-(3.19)]{JM}, \cite[2.9]{Ma})}\label{gamma1} Suppose $q\neq 1$ and (\ref{ss0}) holds. Let $\blam\in\P_n$. The $\gamma$-coefficients $\{\gamma_\t^{(n)}|\t\in\Std(\blam),\blam\in\P_n\}$ are defined to be a multiset of invertible scalars in $K^\times$ which are uniquely determined by:
	\begin{enumerate}
		\item $\gamma_{\t^{\blam}}^{(n)}=[\blam]^!_q\prod\limits_{1\leq s<t\leq\ell}\prod\limits_{1\leq j\leq\lam_i^{(s)}}(q^{j-i}Q_s-Q_t)$; and
		\item if $\s=\t(i,i+1)\triangleright\t$ then $$ \frac{\gamma_{\t}^{(n)}}{\gamma_{\s}^{(n)}}=\frac{(q\res_{\s}(i)-\res_{\t}(i))(\res_{\s}(i)-q\res_{\t}(i))}{(\res_{\s}(i)-\res_{\t}(i))^2}. $$
	\end{enumerate}
\end{dfn}
When the context is clear, we shall omit the superscript ``$(n)$'' and write $\gamma_{\t}$ instead of $\gamma_{\t}^{(n)}$. 

\begin{lem}\label{obobn}{\rm (\cite[Theorems 2.11,2.15, Corollary 2.13]{Ma})}
	Suppose $q\neq 1$, (\ref{ss0}) holds and $R=K$ is a field. Then $$\{\mff_{\s\t}\ |\ \s,\t \in \Std(\blam), \blam\in\P_n\}$$ is a basis of $\HH_{\ell,n}(q,\bQ)$. Moreover,
\begin{enumerate}
\item[1)] if $\s, \t, \u$ and $\v$ are standard tableaux, then $\mff_{\s\t}\mff_{\u\v}=\delta_{\t\u}\gamma_{\t}\mff_{\s\v}$;
\item[2)] if $\blam\in\P_n$, $\s,\t\in\Std(\blam)$ and $1\leq k\leq n$, then $\mff_{\s\t}\mL_k=\res_\t(k)\mff_{\s\t}$, $\mL_k\mff_{\s\t}=\res_\s(k)\mff_{\s\t}$;
\item[3)] for each $\lam\in\P_n$ and $\t\in\Std(\blam)$, $F_{\t}=\frac{1}{\gamma_{\t}}\mff_{\t\t}$ and $F_{\t}$ is a primitive idempotent;
\item[4)] $\{F_\t|\t\in\Std(\blam),\blam\in\P_n\}$ is a complete set of pairwise orthogonal primitive idempotents in $\HH_{\ell,n}(q,\bQ)$.
\end{enumerate}
\end{lem}
We call $\{\mff_{\s\t}\ |\ \s,\t \in \Std(\blam), \blam\in\P_n\}$ the {\bf seminormal basis} of $\HH_{\ell,n}(q,\bQ)$ corresponding to the {\bf cellular basis} $\{\fm_{\s\t}\ |\ \s,\t \in \Std(\blam), \blam\in\P_n\}$ of $\HH_{\ell,n}(q,\bQ)$.\medskip

In the rest of this section, we consider the degenerate cyclotomic Hecke algebra $H_{\ell,n}(\bu)$. First, let's recall the following criteria of semisimplicity for $H_{\ell,n}(\bu)$.

\begin{lem}\text{(\cite[Theorem 6.11]{AMR})} Let $R=K$ be a field. Then $H_{\ell,n}(\bu)$ is semisimple if and only if \begin{equation}\label{ss0b}
(n!)\prod_{\substack{1\leq i<j\leq\ell\\ |d|<n}}(d\cdot 1_K+u_i-u_j)\in K^\times .
\end{equation}
\end{lem}

For any $\t=(\t^{(1)},\cdots,\t^{(\ell)})\in\Std(\blam)$ and any $1\leq k\leq n$, we define \begin{equation}
c_{\t}(k)=j-i+u_c,\quad  \text{if $k$ appears in row $i$ and column $j$ of $\t^{(c)}$}
\end{equation}
We also define $C(k):=\{c_\t(k) | \t\in\Std(\blam), \blam\vdash n\}$.

The condition (\ref{ss0b}) is actually equivalent to the following statement: \begin{equation}\label{separacontent1b}
\begin{matrix}\text{for any $\blam,\bmu\in\P_n$, $\s\in\Std(\blam), \t\in\Std(\bmu)$, if $\s\neq\t$, then there exists}\\
\text{$1\leq k\leq n$ such that $c_\s(k)-c_\t(k)\in K^\times$.}\end{matrix}
\end{equation}

\begin{dfn}\label{Ftb}(\cite{Mur1}, \cite[Definition 6.7]{AMR}) Suppose (\ref{ss0b}) holds. Let $\blam\in\P_n$ and $\t\in\Std(\blam)$. We define
$${\rm F}_{\t}=\prod\limits^n\limits_{k=1}\prod\limits_{\substack{c\in C(k)\\c\neq c_{\t}(k)}}\frac{L_k-c}{c_{\t}(k)-c}.$$
\end{dfn}

For any $\blam\in\P_n$ and $\s,\t\in\Std(\blam)$, we define \begin{equation}\label{seminormalb}
{\rm f}_{\s\t}^{(n)}:={\rm F}_\s {\rm m}_{\s\t}{\rm F}_\t .\end{equation}
When the context is clear, we shall omit the superscript and write ${\rm f}_{\s\t}$ instead of ${\rm f}_{\s\t}^{(n)}$.

\begin{dfn}{\rm (\cite{Mur1}, \cite[Lemma 6.10]{AMR})}\label{gamma1b} Suppose (\ref{ss0b}) holds. Let $\blam\in\P_n$. We define a multiset of elements $\{r^{(n)}_\t\in K^\times|\t\in\Std(\blam),\blam\in\P_n\}$ in $K^\times$ as follows:
	\begin{enumerate}
		\item $r^{(n)}_{\t^{\blam}}=\Bigl(\prod_{l=1}^{\ell}\prod_{i\geq 1}\lam_i^{(l)}!\Bigr)\prod\limits_{1\leq s<t\leq\ell}\prod\limits_{1\leq j\leq\lam_i^{(s)}}(j-i+u_s-u_t)$; and
		\item if $\s=\t(i,i+1)\triangleright\t$ then $$ \frac{r^{(n)}_{\t}}{r^{(n)}_{\s}}=\frac{(1+c_{\s}(i)-c_{\t}(i))(c_{\s}(i)-c_{\t}(i)-1)}{(c_{\s}(i)-c_{\t}(i))^2}. $$
	\end{enumerate}
\end{dfn}
When the context is clear, we shall omit the superscript ``$(n)$'' and write $r_{\t}$ instead of $r_{\t}^{(n)}$.

\begin{lem}\label{obobn2}{\rm (\cite[Proposition 3.4]{Ma2})}
	Suppose (\ref{ss0b}) holds and $R=K$ is a field. Then $$\{{\rm f}_{\s\t}\ |\ \s,\t \in \Std(\blam), \blam\in\P_n\}$$ is a basis of $H_{\ell,n}(\bu)$. Moreover,
\begin{enumerate}
\item[1)] if $\s, \t, \u$ and $\v$ are standard tableaux, then ${\rm f}_{\s\t}{\rm f}_{\u\v}=\delta_{\t\u}r_{\t}{\rm f}_{\s\v}$;
\item[2)] if $\blam\in\P_n$, $\s,\t\in\Std(\blam)$ and $1\leq k\leq n$, then ${\rm f}_{\s\t}L_k=c_\t(k){\rm f}_{\s\t}$, $L_k{\rm f}_{\s\t}=c_\s(k){\rm f}_{\s\t}$;
\item[3)] for each $\lam\in\P_n$ and $\t\in\Std(\blam)$, ${\rm F}_{\t}=\frac{1}{r_{\t}}{\rm f}_{\t\t}$ and ${\rm F}_{\t}$ is a primitive idempotent;
\item[4)] $\{{\rm F}_\t|\t\in\Std(\blam),\blam\in\P_n\}$ is a complete set of pairwise orthogonal primitive idempotents in $H_{\ell,n}(\bu)$.
\end{enumerate}
\end{lem}
We call $\{{\rm f}_{\s\t}\ |\ \s,\t \in \Std(\blam), \blam\in\P_n\}$ the {\bf seminormal basis} of $H_{\ell,n}(\bu)$ corresponding to the {\bf cellular basis} $\{{\rm m}_{\s\t}\ |\ \s,\t \in \Std(\blam), \blam\in\P_n\}$ of $H_{\ell,n}(\bu)$.


\bigskip

\section{The non-degenerate case}

In this section we shall only consider the non-degenerate cyclotomic Hecke algebra $\HH_{\ell,n}(q,\bQ)$. Our purpose is to give the proof of the main results Theorem \ref{mainthm1} and Theorem \ref{mainthm2}.
Throughout this section, we assume that $R=K$ is a field, $q\neq 1$ and (\ref{ss0}) holds. In particular, this implies that $\HH_{\ell,n}(q,\bQ)$ is (split) semisimple over $K$.

Let $\{\mathfrak{m}_{\s\t}|\s,\t\in\Std(\blam),\blam\in\P_n\}$ be the Dipper-James-Mathas cellular basis, and $\{\mff_{\s\t}|\s,\t\in\Std(\blam),\blam\in\P_n$ be the corresponding seminormal basis of $\HH_{\ell,n}(q,\bQ)$. For each $\lam\in\P_n$, we define $$
\HH_{\ell,n}^{\rhd\blam}:=\text{$K$-Span}\{\fm_{\s\t}|\s,\t\in\Std(\bmu), \blam\lhd\bmu\in\P_n\},
$$
which is a cell ideal of $\HH_{\ell,n}(q,\bQ)$ with respect to the cellular basis. For any $1\leq k\leq n$ and $\s,\t\in\Std(\blam)$, we have that (\cite[(2.3)]{Ma}) \begin{equation}\label{L1}
\fm_{\s\t}\mL_k=\res_{\t}(k)\fm_{\s\t}+\sum_{\substack{\v\in\Std(\blam)\\ \v\triangleright\t}} a_{\v}\fm_{\s\v}\pmod{\HH_{\ell,n}^{\rhd\blam}},
 \end{equation}
where $a_\v\in K$ for each $\t\lhd\v\in\Std(\blam)$.

Mathas has yet introduced in \cite[\S3]{Ma} another cellular basis which will be called {\bf the dual cellular basis} of $\HH_{\ell,n}(q,\bQ)$. We now recall his construction.

For each $\blam\in\P_n$, we define  \begin{equation}
\fn_{\t_\blam\t_\blam}:=\Bigl(\sum_{w\in\Sym_{\blam'}}(-q)^{-\ell(w)}T_w\Bigr)\Bigl(\prod_{s=2}^{\ell}\prod_{k=1}^{|\lam^{(\ell)}|+|\lam^{(\ell-1)}|+\cdots+|\lam^{(\ell-s+2)}|}(\mL_k-Q_{\ell-s+1})\Bigr).
\end{equation}
If $\t\in\Std(\blam)$, then we define $d'(\t)\in\Sym_n$ by $\t_\blam d'(\t)=\t$. For any  $\s,\t\in\Std(\blam)$, we set \begin{equation}\label{nst}
\fn_{\s\t}=(-q)^{-\ell(d'(\s))-\ell(d'(\t))}T_{d'(\s)}^*\fn_{\t_\blam\t_\blam}T_{d'(\t)}. \end{equation}

\begin{thm}\text{(\cite{Ma})} With respect to the opposite poset $(\P_n, \unlhd)$ and the anti-involution $\ast$, the set $\{\fn_{\s\t}|\s,\t\in\Std(\blam),\blam\in\P_n\}$ forms a cellular basis of $\HH_{\ell,n}(q,\bQ)$.
\end{thm}

We call it the {\bf dual cellular basis} of $\HH_{\ell,n}(q,\bQ)$. For each $\lam\in\P_n$, we define $$
\check{\HH}_{\ell,n}^{\lhd\blam}:=\text{$K$-Span}\{\fn_{\s\t}|\s,\t\in\Std(\bmu), \blam\rhd\bmu\in\P_n\},
$$
which is a cell ideal of $\HH_{\ell,n}(q,\bQ)$ with respect to the dual cellular basis. For any $1\leq k\leq n$ and $\s,\t\in\Std(\blam)$, we have that (\cite[Proposition 3.3]{Ma}) \begin{equation}\label{L2}
\fn_{\s\t}\mL_k=\res_{\t}(k)\fn_{\s\t}+\sum_{\substack{\v\in\Std(\blam)\\ \v\lhd\t}} b_{\v}\fn_{\s\v} \pmod{\check{\HH}_{\ell,n}^{\lhd\blam}}, \end{equation}
where $b_\v\in K$ for each $\t\rhd\v\in\Std(\blam)$.

%



\begin{dfn}\label{gdfn} Let $\blam\in\P_n$. For any $\s,\t\in\Std(\blam)$, we define $$\mfg_{\s\t}:=F_{\s}\fn_{\s\t}F_{\t}. $$
\end{dfn}

\begin{rem}\label{convention} Note that our notations $\fn_{\s\t}, \mfg_{\s\t}$ differ with the corresponding notations in \cite{Ma} by a conjugation and an invertible scalar. The elements $\fn_{\s\t}, g_{\s\t}$ in the current paper should be identified with the elements $n_{\s'\t'}, g_{\s'\t'}$ in \cite{Ma} up to some invertible scalar. In particular, our dual cellular basis $\{\fn_{\s\t}\}$ use the partial order ``$\unlhd$'', while \cite{Ma} use the partial order ``$\unrhd$'' for the dual cellular basis. Our convention for the notations $\fn_{\s\t}$ in this paper agrees with the one used in \cite[Section 3]{HuMathas:TAMS}.
\end{rem}

\begin{dfn}\text{(\cite[\S3]{Ma})}\label{prime} Suppose $\hat{q}, \hat{Q}_1,\cdots,\hat{Q}_\ell$ are indeterminates over $\Z$. Set $\mathscr{A}:=\Z[\hat{q}^{\pm 1}, \hat{Q}_1,\cdots,\hat{Q}_\ell]$. Let $\mathscr{K}:=\Q(\hat{q}, \hat{Q}_1,\cdots,\hat{Q}_\ell)$ be the rational functional field on $\hat{q}, \hat{Q}_1,\cdots,\hat{Q}_\ell$.
Let $\HH_{\ell,n}(\hat{q},\hat{\bQ})$ be the non-degenerate cyclotomic Hecke algebra of type $G(\ell,1,n)$ over $\mathscr{A}$ with Hecke parameter $\hat{q}$ and cyclotomic parameters $\hat{\bQ}:=(\hat{Q}_1,\cdots,\hat{Q}_\ell)$.
Set $\HH^{\mathscr{K}}_{\ell,n}(\hat{q},\hat{\bQ}):=\mathscr{K}\otimes_{\mathscr{A}}\HH_{\ell,n}(\hat{q},\hat{\bQ})$. Then $\HH^{\mathscr{K}}_{\ell,n}(\hat{q},\hat{\bQ})$ is split semisimple.
We set $'$ to be the unique ring involution  of $\HH_{\ell,n}(\hat{q},\hat{\bQ})$ (\cite[\S3]{Ma}) which is defined on generators by $$
{T}'_0:={T}_0,\,\, {T}'_i:=-\hat{q}^{-1}{T}_i,\,\, \hat{q}':=\hat{q}^{-1},\,\,\hat{Q}'_j:=\hat{Q}_{\ell-j+1},\quad 1\leq i<n,\, 1\leq j\leq\ell .
$$
\end{dfn}
Clearly, $'$ naturally extends to a ring involution  of $\HH^{\mathscr{K}}_{\ell,n}(\hat{q},\hat{\bQ})$. We have $\mL'_m=\mL_m$ for any $1\leq m\leq n$, and $\fm'_{\s\t}=\fn_{\s'\t'}$, $(\res_\t(k))'=\res_{\t'}(k)$ for any $1\leq k\leq n$ by \cite[(3.2)]{Ma}. It follows from Definition \ref{Ft} that \begin{equation}\label{2primes}
F'_\t=F_{\t'},\quad \mff'_{\s\t}=(F_{\s}\fm_{\s\t}F_{\t})'=F'_{\s}\fm'_{\s\t}F'_{\t}=F_{\s'}\fn_{\s'\t'}F_{\t'}=\mfg_{\s'\t'}.
\end{equation}
For any rational function $f$ on $\hat{q},\hat{Q}_1,\cdots,\hat{Q}_\ell$, we use $f'$ to denote the rational function obtained from $f$ by substituting  $\hat{q}$ and $\hat{Q}_i$ (for $1\leq i\leq\ell$) with $-\hat{q}^{-1}$ and $\hat{Q}_{\ell-i+1}$ respectively.
By Definition \ref{gamma1}, for each $\t\in\Std(\blam)$, the scalar $\gamma_\t$ is given by the evaluation of a rational function $\gamma_\t(\hat{q},\hat{Q}_1,\cdots,\hat{Q}_\ell)$ at $\hat{q}:=q,\hat{Q}_i:=Q_i, 1\leq i\leq\ell$. Thus the notation $$\gamma'_\t:=1_K\otimes_{\mathscr{A}}\gamma'_{\t}(\hat{q},\hat{Q}_1,\cdots,\hat{Q}_\ell)\in K^\times $$ does make sense.

Note that in general we have $\gamma'_\t\neq\gamma_{\t'}$. For example, if $\ell=1=Q_1$, $\lam=(2,1)$, $\t=\t^\lam s_2$, then $$
\gamma_\t=\frac{(q^2-q^{-1})(q-1)(1+q)}{(q-q^{-1})^2},\quad \gamma'_\t=\frac{(q^{-2}-q)(q^{-1}-1)(1+q^{-1})}{(q-q^{-1})^2}\neq\gamma_{\t'}=1+q .
$$

\begin{cor}\label{obobn2cor}
	Suppose $q\neq 1$, (\ref{ss0}) holds and $R=K$ is a field. Then \begin{equation}\label{dualsemi}\{\mfg_{\s\t}\ |\ \s,\t \in \Std(\blam), \blam\in\P_n\}\end{equation} is a basis of $\HH_{\ell,n}(q,\bQ)$. Moreover,
\begin{enumerate}
\item[1)] if $\s, \t, \u$ and $\v$ are standard tableaux, then $\mfg_{\s\t}\mfg_{\u\v}=\delta_{\t\u}\gamma'_{\t'}\mfg_{\s\v}$;
\item[2)] if $\blam\in\P_n$, $\s,\t\in\Std(\blam)$ and $1\leq k\leq n$, then $\mfg_{\s\t}\mL_k=\res_\t(k)\mfg_{\s\t}$, $\mL_k\mfg_{\s\t}=\res_\s(k)\mfg_{\s\t}$;
\item[3)] for each $\lam\in\P_n$ and $\t\in\Std(\blam)$, $F_{\t}=\mfg_{\t\t}/\gamma'_{\t'}$;
\end{enumerate}
\end{cor}
\begin{proof} This follows from (\ref{2primes}) and Lemma \ref{obobn}.
\end{proof}

We call (\ref{dualsemi}) the {\bf dual seminormal basis} of $\HH_{\ell,n}(q,\bQ)$ corresponding to the dual cellular basis $\{\fn_{\s\t}|\s,\t\in\Std(\blam),\blam\in\P_n\}$.

\begin{lem}{\rm (\cite[Remark 3.6]{Ma})}\label{cor2} Let $\blam$ be a multipartition of $n$ and $\s,\t\in \Std(\blam)$. Suppose $q\neq 1$, (\ref{ss0}) holds and $R=K$ is a field. Then
	\begin{enumerate}
		\item[1)] For any standard tableau $\t$, we have
		$$\mfg_{\t'\t'}=\mff'_{\t\t}=\gamma_{\t}'F_{\t'}=\frac{\gamma_{\t}'}{\gamma_{\t'}}\mff_{\t'\t'}.$$
		\item[2)] There exists a unique scalar $\alpha_{\s\t}\in K^\times$ such that $\mfg_{\s\t}=\alpha_{\s\t}\mff_{\s\t}$. Moreover, $\alpha_{\s\t}^2=\gamma_{\s'}'\gamma_{\t'}'/\gamma_{\s}\gamma_{\t}$.
	\end{enumerate}
\end{lem}
\begin{proof} Part 1) follows from (\ref{2primes}), Corollary \ref{obobn2cor} 3) and Definition \ref{gdfn}. For Part 2), on the one hand, combining Lemma \ref{obobn} 2), Corollary \ref{obobn2cor} 2) with (\ref{separacontent1}), we can deduce that $\alpha_{\s\t}:=\mfg_{\s\t}/\mff_{\s\t}\in K^\times$.
On the other hand, applying the anti-involution ``$*$'', we can get that $\mfg_{\t\s}=\alpha_{\s\t}\mff_{\t\s}$ and hence $\alpha_{\s\t}=\alpha_{\t\s}$. Therefore, $$
\gamma'_{\t'}\mfg_{\s\s}=\mfg_{\s\t}g_{\t\s}=\alpha_{\s\t}^2\mff_{\s'\t'}\mff_{\t'\s'}=\alpha_{\s\t}^2\gamma_{\t'} \mff_{\s'\s'}.
$$
By 1) we have $\mfg_{\s\s}/ \mff_{\s'\s'}=\gamma'_{\s'}/\gamma_{\s}$. Hence Part 2) of the lemma follows.
\end{proof}

\begin{rem} Suppose that $q, Q_1,\cdots,Q_\ell$ are indeterminates over $\Z$. Then by (\ref{2primes}) we have \begin{equation}\label{fprimest}
\mff'_{\s\t}=\mfg_{\s'\t'}=\alpha_{\s'\t'}\mff_{\s'\t'},
\end{equation}
for any $\s,\t\in\Std(\blam)$ and $\blam\in\P_n$. Note that the scalar $\alpha_{\s\t}$ in our paper should be identified with the scalar $\alpha_{\s'\t'}$ in the notation of \cite{Ma}. In view of our convention of notations, we have that $\mfg_{\s\t}=\alpha_{\s\t}\mff_{\s\t}$, while in view of the convention of notations in \cite{Ma}, we have $\mfg_{\s\t}=\alpha_{\s\t}\mff_{\s'\t'}$. It follows from Lemma \ref{cor2} that $\gamma_{\s'}'\gamma_{\t'}'/\gamma_{\s}\gamma_{\t}$ always has a square root in $K^\times$ which is a rational function on $q,Q_1,\cdots,Q_\ell$. In \cite[Remark 3.6]{Ma} Mathas has asked whether one can give an intrinsic explanation of this fact and in particular determine the sign of each scalar $\alpha_{\s\t}$. In this paper will present some explicit combinatorial formulae for these scalars $\alpha_{\s\t}$ as some rational functions on $q,Q_1,\cdots,Q_\ell$ and affirmatively answer Mathas's above question.
\end{rem}

For the reader's convenience, we include below a lemma which gives a recursive formula for the $\gamma'$-coefficients associated to the dual seminormal bases.

\begin{lem}
	 Suppose $q\neq 1$ and (\ref{ss0}) holds. Let $\blam\in\P_n$. The coefficients of the dual seminormal basis $\{\mfg_{\s\t}\ |\ \s,\t \in \Std(\blam), \blam\in\P_n\}$ can be uniquely determined by:
	\begin{enumerate}
		\item $\gamma'_{(\t_\blam)'}=\gamma'_{\t^{\blam'}}=q^{C}[\blam']^!_q\prod\limits_{1\leq t<s\leq\ell}\prod\limits_{1\leq j\leq\lam_i^{(s)}}(q^{j-i}Q_s-Q_t)$, where $C=-\sum\limits_{c=1}^\ell\sum\limits_{i\geq1}\frac{(\lambda^{(c)'})_i((\lambda^{(c)'})_i-1)}{2}$; and
		\item if $\s=\t(i,i+1)\triangleleft\t$ then $$ \frac{\gamma'_{\t'}}{\gamma'_{\s'}}=q^{-2}\frac{(q\res_{\s}(i)-\res_{\t}(i))(\res_{\s}(i)-q\res_{\t}(i))}{(\res_{\t}(i)-\res_{\s}(i))^2}=q^{-2}\frac{\gamma_{\s}}{\gamma_{\t}}. $$
	\end{enumerate}
\end{lem}

\begin{proof} This follows from Definition \ref{prime} and the equality $(\res_{\t}(k))'=\res_{\t'}(k)$.
\end{proof}

\begin{lem}{\rm (\cite[Proposition 2.7]{Ma})}\label{seminormalB} Suppose $q\neq 1$, (\ref{ss0}) holds and $R=K$ is a field. Let $\blam\in\P_n$ and $\s,\u\in\Std(\blam)$. Let $i$ be an integer with $1\leq i< n$ and $\t:=\s(i,i+1)$. If $\t$ is standard then
	$$\mff_{\u\s}T_i=\begin{cases*}
		A_i(\s)\mff_{\u\s}+\mff_{\u\t}, &\text{if $\t\triangleleft\s$},\\
		A_i(\s)\mff_{\u\s}+B_i(\s)\mff_{\u\t}, &\text{if $\s\triangleleft\t$},
	\end{cases*}$$
where $$
A_i(\s)=\frac{(q-1)\res_\s(i+1)}{\res_\s(i+1)-\res_\s(i)},\quad B_i(\s):=\frac{\gamma_\s}{\gamma_\t}=\frac{(q\res_{\s}(i)-\res_{\s}(i+1))(\res_{\s}(i)-q\res_{\s}(i+1))}{(\res_{\s}(i+1)-\res_{\s}(i))^2}.
$$
If $\t$ is not standard then
	$$\mff_{\u\s}T_i=\begin{cases*}
		q\mff_{\u\s}, &\text{if $i$ and $i+1$ are in the same row of $\s$},\\
		-\mff_{\u\s}, &\text{if $i$ and $i+1$ are in the same column of $\s$}.
	\end{cases*}$$
\end{lem}

Let $\HH_q(\Sym_n)$ be the Iwahori-Hecke algebra of the symmetric group $\Sym_n$, which can be identified with the $K$-subalgebra of $\HH_{\ell,n}(q,\bQ)$ generated by $T_1,\cdots,T_{n-1}$.

\begin{lem}{\rm (\cite{Mur3}, \cite[Proposition 4.1, Lemma 4.3]{Ma})}\label{M5prop}
	Suppose $q\neq 1$, (\ref{ss0}) holds and $R=K$ is a field. Let $\blam\in\P_n$ and $i$ an integer with $1\leq i<n$. Then there exist invertible elements $\{\Phi_{\t}|\t\in\Std(\blam)\}$ in $\HH_q(\Sym_n)$ such that
	\begin{enumerate}
		\item[(i)] for any $\s,\t\in\Std(\blam)$, $\mff_{\s\t}=\Phi_\s^*\mff_{\t^\blam\t^\blam}\Phi_\t$;
		\item[(ii)] $\Phi_{\t^\blam}=1$, and if $\s:=\t(i,i+1)\lhd\t$, then $$\Phi_\s=\Phi_\t(T_i-A_i(\t)).
$$
	\end{enumerate}
\end{lem}


\begin{lem}\label{phiprime} Let $\blam\in\P_n$ and $\t\in\Std(\blam)$. Let $i$ be an integer with $1\leq i<n$. Suppose $q, Q_1,\cdots,Q_\ell$ are indeterminates over $\Z$.
If $\s:=\t(i,i+1)\in\Std(\blam)$ with $\s\lhd\t$, then $\Phi_\s'=(-q)^{-1}\Phi_\t'(T_i-A_i(\t'))$.
\end{lem}

\begin{proof} Recall the ring involution $'$ introduced in Definition \ref{prime}, which is defined on generators by $$
T'_0:=T_0,\,\, T'_i:=-q^{-1}T_i,\,\, q':=q^{-1},\,\,Q'_j:=Q_{\ell-j+1},\quad 1\leq i<n,\, 1\leq j\leq\ell .
$$
It follows from Lemma \ref{M5prop} that $\Phi'_\s=\Phi'_\t(-q^{-1}T_i-A_i(\t)')$. Thus it suffices to show that $A_i(\t)'=-q^{-1}A_i(\t')$.

Since $q'=q^{-1}$, we have that
	$$\begin{aligned}
A_i(\t)'&=\big(\frac{(q-1)\res_\t(i+1)}{\res_\t(i+1)-\res_\t(i)}\big)'
			=\frac{(q^{-1}-1)\res_\t(i+1)'}{\res_\t(i+1)'-\res_\t(i)'}\\
	&=\frac{(q^{-1}-1)\res_{\t'}(i+1)}{\res_{\t'}(i+1)-\res_{\t'}(i)}\\
&=(-q)^{-1}\frac{(q-1)\res_{\t'}(i+1)}{\res_{\t'}(i+1)-\res_{\t'}(i)}=(-q)^{-1}A_i(\t').\end{aligned}$$
This completes the proof of the lemma.
\end{proof}


\begin{lem}\label{scalaralt} Let $\blam\in\P_n$ and $\t\in\Std(\blam)$. Suppose $q, Q_1,\cdots,Q_\ell$ are indeterminates over $\Z$. Then we have $$
\mff_{\t_{\blam'}\t_{\blam'}}\Phi'_{\t}=(-q)^{-\ell(d(\t))}\frac{\gamma_{\t_{\blam'}}}{\gamma_{\t'}}\mff_{\t_{\blam'}\t'},\quad
(\Phi_{\t}^*)'\mff_{\t_{\blam'}\t_{\blam'}}=(-q)^{-\ell(d(\t))}\frac{\gamma_{\t_{\blam'}}}{\gamma_{\t'}}\mff_{\t'\t_{\blam'}} .
$$
\end{lem}

\begin{proof} Recall that $d(\t)\in\Sym_n$ such that $\t^{\blam}d(\t)=\t$. Fix a reduced expression $d(\t)=s_{i_1}\cdots s_{i_l}$, where $1\leq i_j<n$ for each $j$.
For each $1\leq k\leq l$, we define $w_k=s_{i_1}s_{i_2}\cdots s_{i_k}$ and set $\t_k:=\t^{\blam}w_k$, $\t_0=\t^{\blam}$. Then $\t_{l}=\t$. We get the following
sequence of standard $\blam$-tableaux:
$$\t^{\blam}=\t_0\triangleright\t_1\triangleright\t_2\triangleright\cdots\triangleright\t_l=\t.$$
Combining this with Lemmas \ref{M5prop} and \ref{phiprime}, we get that
\begin{equation}\label{phit}\Phi'_{\t}=(-q)^{-\ell(d(\t))}(T_{i_1}-A_{i_1}(\t_0'))(T_{i_2}-A_{i_2}(\t_1'))\cdots(T_{i_l}-A_{i_l}(\t_{l-1}')).\end{equation}

Note that $\t_{k-1}'\lhd\t_k'=\t'_{k-1}s_{i_k}$ for each $1\leq k\leq l$. We get the following sequence of standard $\blam'$-tableaux:
	$$\t_{\blam'}=\t_0'\triangleleft\t_1'\triangleleft\t_2'\triangleleft\cdots\triangleleft\t_l'=\t'.$$	
Applying Lemma \ref{seminormalB}, Lemma \ref{M5prop} and (\ref{phit}), we get that
	$$\begin{aligned}
	\mff_{\t_{\blam'}\t_{\blam'}}\Phi'_{\t}
	&=(-q)^{-\ell(d(\t))}\mff_{\t_{\blam'}\t_{\blam'}}(T_{i_1}-A_{i_1}(\t_0'))(T_{i_2}-A_{i_2}(\t_1'))\cdots(T_{i_l}-A_{i_l}(\t_{l-1}'))\qquad \text{(by (\ref{phit}))}\\
&=(-q)^{-\ell(d(\t))}\frac{\gamma_{\t_0'}}{\gamma_{\t_1'}}\frac{\gamma_{\t_1'}}{\gamma_{\t_2'}}\cdots\frac{\gamma_{\t_{l-1}'}}{\gamma_{\t_l'}}\mff_{\t_{\blam'}\t'}\qquad \text{(by Lemmas \ref{seminormalB} and \ref{M5prop})}\\
	&=(-q)^{-\ell(d(\t))}\frac{\gamma_{\t_{\blam'}}}{\gamma_{\t'}}\mff_{\t_{\blam'}\t'}.
	\end{aligned}$$
	Applying the anti-automorphism $*$ and noting that $*$ commutes with $'$, we get that $$(\Phi_{\t}^*)'\mff_{\t_{\blam'}\t_{\blam'}}=(-q)^{-\ell(d(\t))}\frac{\gamma_{\t_{\blam'}}}{\gamma_{\t'}}\mff_{\t'\t_{\blam'}}. $$
This completes the proof of the lemma.
\end{proof}

The following result reveals some hidden relationship between $\gamma_{\t'}, \gamma_{\t}', \gamma_{\t_{\blam'}}$ and $\gamma_{\t^{\blam}}'$.

\begin{lem}\label{gam'tgamt'} Suppose $q\neq 1$, (\ref{ss0}) holds and $R=K$ is a field.
	Let $\blam\in\P_n$ be a multipartition of $n$ and $\t\in\Std(\blam)$. Then we have that
	$$\gamma_{\t'}\gamma_{\t}'=q^{-2\ell(d(\t))}\gamma_{\t_{\blam'}}\gamma_{\t^{\blam}}'.$$
\end{lem}
	
\begin{proof} Without loss of generality we can assume that $q, Q_1,\cdots,Q_\ell$ are indeterminates over $\Z$. As in the proof of Lemma \ref{scalaralt}, we fix a reduced expression $d(\t)=s_{i_1}\cdots s_{i_l}$, where $1\leq i_j<n$ for each $j$.
For each $1\leq k\leq l$, we define $w_k=s_{i_1}s_{i_2}\cdots s_{i_k}$ and set $\t_k:=\t^{\blam}w_k$, $\t_0=\t^{\blam}$. Then we get the following two sequences of standard tableaux:
	$$
	\t^{\blam}=\t_0\triangleright\t_1\triangleright\t_2\triangleright\cdots\triangleright\t_l=\t,\quad\,
	\t_{\blam'}=\t_0'\triangleleft\t_1'\triangleleft\t_2'\triangleleft\cdots\triangleleft\t_l'=\t'.
	$$
	By definition, we have
	$$\gamma_{\t}=\gamma_{\t^{\blam}}\frac{\gamma_{\t_1}}{\gamma_{\t_0}}\frac{\gamma_{\t_2}}{\gamma_{\t_1}}\cdots\frac{\gamma_{\t_l}}{\gamma_{\t_{l-1}}}.$$
	Applying the ring involution $'$, we get that
$$\gamma_{\t}'=\gamma_{\t^{\blam}}'\big(\frac{\gamma_{\t_1}}{\gamma_{\t_0}}\big)'\big(\frac{\gamma_{\t_2}}{\gamma_{\t_1}}\big)'\cdots\big(\frac{\gamma_{\t_l}}{\gamma_{\t_{l-1}}}\big)'.$$
	For each $1\leq k\leq l$, by Definition \ref{gamma1} and Definition \ref{prime} we have that
	$$\begin{aligned}
	\big(\frac{\gamma_{\t_k}}{\gamma_{\t_{k-1}}}\big)'
	&=\bigg(\frac{(q\res_{\t_{k-1}}(i_k)-\res_{\t_k}(i_k))(\res_{\t_{k-1}}(i_k)-q\res_{\t_k}(i_k))}{(\res_{\t_{k-1}}(i_k)-\res_{\t_k}(i_k))^2}\bigg)'\\
	&=\frac{(q^{-1}\res_{\t_{k-1}'}(i_k)-\res_{\t_k'}(i_k))(\res_{\t_{k-1}'}(i_k)-q^{-1}\res_{\t_k'}(i_k))}{(\res_{\t_{k-1}'}(i_k)-\res_{\t_k'}(i_k))^2}\\
	&=q^{-2}\frac{(\res_{\t_{k-1}'}(i_k)-q\res_{\t_k'}(i_k))(q\res_{\t_{k-1}'}(i_k)-\res_{\t_k'}(i_k))}{(\res_{\t_{k-1}'}(i_k)-\res_{\t_k'}(i_k))^2}\\
	&=q^{-2}\frac{\gamma_{\t_{k-1}'}}{\gamma_{\t_k'}}.
	\end{aligned}$$
	Hence, we can get that
	$$\begin{aligned}
	\gamma_{\t}'
&=\gamma_{\t^{\blam}}'\big(\frac{\gamma_{\t_1}}{\gamma_{\t_0}}\big)'\big(\frac{\gamma_{\t_2}}{\gamma_{\t_1}}\big)'\cdots\big(\frac{\gamma_{\t_l}}{\gamma_{\t_{l-1}}}\big)'
=q^{-2\ell(d(\t))}\gamma_{\t^{\blam}}'\frac{\gamma_{\t_{0}'}}{\gamma_{\t_1'}}\frac{\gamma_{\t_{1}'}}{\gamma_{\t_2'}}\cdots\frac{\gamma_{\t_{l-1}'}}{\gamma_{\t_l'}}\\
	&=q^{-2\ell(d(\t))}\gamma_{\t^{\blam}}'\frac{\gamma_{\t_{\blam'}}}{\gamma_{\t'}}.
	\end{aligned}$$
	It follows that $\gamma_{\t'}\gamma_{\t}'=q^{-2\ell(d(\t))}\gamma_{\t_{\blam'}}\gamma_{\t^{\blam}}'$. This completes the proof of the lemma.
\end{proof}

Let $\blam\in\P_n$ and $\s,\t\in\Std(\blam)$. Recall that $\alpha_{\s\t}\in K^\times$ is an invertible scalar introduced in Lemma \ref{cor2} such that $\mfg_{\s\t}=\alpha_{\s\t}\mff_{\s\t}$. Now we can give the proof of the first main result of this paper which presents some explicit formulae for the scalar $\alpha_{\s\t}$.


\medskip
\noindent
{\textbf{Proof of Theorem \ref{mainthm1}}}: To prove the theorem, we can assume without loss of generality that $q, Q_1,\cdots,Q_\ell$ are indeterminates over $\Z$. In this case, we can use the ring involution $'$ of $\HH_{\ell,n}(q,\bQ)$ introduced in Definition \ref{prime}. By Lemma \ref{M5prop}, we have that $\mff_{\s'\t'}=\Phi_{\s'}^*\mff_{\t^{\blam'}\t^{\blam'}}\Phi_{\t'}$. Applying the involution $'$ and using Lemma \ref{cor2} 1), we can get that
\begin{equation}\label{gst1}
\mfg_{\s\t}=\mff'_{\s'\t'}=(\Phi_{\s'}^*\mff_{\t^{\blam'}\t^{\blam'}}\Phi_{\t'})'=(\Phi_{\s'}^*)'\mff_{\t^{\blam'}\t^{\blam'}}'\Phi'_{\t'}
=\frac{\gamma_{\t^{\blam'}}'}{\gamma_{\t_{\blam}}}(\Phi_{\s'}^*)'\mff_{\t_{\blam}\t_{\blam}}\Phi'_{\t'}.\end{equation}

Applying Lemma \ref{scalaralt}, we can deduce that
	$$\begin{aligned}
(\Phi_{\s'}^*)'\mff_{\t_{\blam}\t_{\blam}}\Phi_{\t'}'&=(-q)^{-\ell(d(\s'))}\frac{\gamma_{\t_\blam}}{\gamma_\s}\mff_{\s\t_\blam}\Phi_{\t'}'\\
&=(-q)^{-\ell(d(\s'))}\frac{1}{\gamma_\s}\mff_{\s\t_\blam}\mff_{\t_\blam\t_{\blam}}\Phi_{\t'}'\\
&=(-q)^{-\ell(d(\s'))-\ell(d(\t'))}\frac{\gamma_{\t_{\blam}}}{\gamma_{\s}\gamma_{\t}}\mff_{\s\t_\blam}\mff_{\t_\blam\t}
=(-q)^{-\ell(d(\s'))-\ell(d(\t'))}\frac{\gamma^2_{\t_{\blam}}}{\gamma_{\s}\gamma_{\t}}\mff_{\s\t}.
\end{aligned}
$$
Combining this with (\ref{gst1}), we can deduce that $$
\mfg_{\s\t}=(-q)^{-\ell(d(\s'))-\ell(d(\t'))}\frac{\gamma_{\t_{\blam}}\gamma'_{\t^{\blam'}}}{\gamma_{\s}\gamma_{\t}}\mff_{\s\t}.
$$
Hence $\alpha_{\s\t}=(-q)^{-\ell(d(\s'))-\ell(d(\t'))}\frac{\gamma_{\t_{\blam}}\gamma'_{\t^{\blam'}}}{\gamma_{\s}\gamma_{\t}}$. This proves the first equality of the theorem.

Finally, by Lemma \ref{gam'tgamt'}, we have that
	$$\frac{\gamma_{\t_{\blam}}}{\gamma_{\s}}=q^{2\ell(d(\s'))}\frac{\gamma_{\s'}'}{\gamma_{\t^{\blam'}}'},\qquad
\frac{\gamma_{\t^{\blam'}}'}{\gamma_{\t}}=q^{2\ell(d(\t))}\frac{\gamma_{\t'}'}{\gamma_{\t_{\blam}}}.$$
It follows that $$
(-q)^{-\ell(d(\s'))-\ell(d(\t'))}\frac{\gamma_{\t_{\blam}}\gamma_{\t^{\blam'}}'}{\gamma_{\s}\gamma_{\t}}=
(-q)^{\ell(d(\s'))+\ell(d(\t'))}\frac{\gamma_{\s'}'\gamma_{\t'}'}{\gamma_{\t_{\blam}}\gamma_{\t^{\blam'}}'},$$
which proves the second equality of the theorem.
\qed
\medskip

Let $\iota: \HH_{\ell,n-1}(q,\bQ)\hookrightarrow\HH_{\ell,n}(q,\bQ)$ be the natural inclusion which is defined on generators by $\iota(T_i):=T_i$ for $0\leq i<n-1$. In order to avoid the confusion between the notations for $\HH_{\ell,n-1}(q,\bQ)$ and $\HH_{\ell,n}(q,\bQ)$. We add a superscript $(n)$ to indicate that it is the notation for $\HH_{\ell,n}(q,\bQ)$. Let $\bmu\in\P_{n-1}$ and $\s,\t\in\Std(\blam)$. Then we have \begin{equation}
\mff_{\s\t}^{(n-1)}=\sum_{\blam\in\P_n}\sum_{\u,\v\in\Std(\blam)}\beta_{\u\v}^{\s\t}\mff_{\u\v}^{(n)},
\end{equation}
where $\beta_{\u\v}^{\s\t}\in K$ for each pair $(\u,\v)$. In the rest of this section, we shall give some explicit formulae for these scalars $\beta_{\u\v}^{\s\t}$.

\begin{lem}\label{fst3} Suppose $q\neq 1$, (\ref{ss0}) holds and $R=K$ is a field. Let $\bmu\in\P_{n-1}, \blam\in\P_n$, and $\s,\t\in\Std(\bmu)$, $\u,\v\in\Std(\blam)$. Then \begin{enumerate}
\item $\beta_{\u\v}^{\s\t}\neq 0$ only if $\u\downarrow_{n-1}=\s$ and $\v\downarrow_{n-1}=\t$;
\item $\beta_{\u\u}^{\s\s}\neq 0$ if and only if $\u\downarrow_{n-1}=\s$. In that case, $\beta_{\u\u}^{\s\s}=\gamma^{(n-1)}_{\s}/\gamma^{(n)}_{\u}$.
\end{enumerate}
\end{lem}

\begin{proof} Without loss of generality we can assume $n\geq 2$. Suppose that $\beta_{\u\v}^{\s\t}\neq 0$. Then the equalities $\u\downarrow_{n-1}=\s$ and $\v\downarrow_{n-1}=\t$ follows from (\ref{separacontent1}) and Lemma \ref{obobn} 2) by considering the left and the right actions of $\mL_m$ for $1\leq m\leq n$. This proves the part a) of the lemma.

Let $\bmu\in\P_{n-1}$ and $\s\in\Std(\bmu)$. By the part a) of the lemma, we can write \begin{equation}
\mff_{\s\s}^{(n-1)}=\sum_{\blam\in\P_n}\sum_{\substack{\u,\v\in\Std(\blam)\\ \u\downarrow_{n-1}=\s=\v\downarrow_{n-1}}}\beta_{\u\v}^{\s\s}\mff_{\u\v}^{(n)}.
\end{equation}

Suppose $\beta_{\u\v}^{\s\s}\neq 0$. Then by Lemma \ref{fst3} we see that $\u\downarrow_{n-1}=\s=\v\downarrow_{n-1}$. Since ${\rm{Shape}}(\u)={\rm{Shape}}(\v)$, it follows that $\u=\v$. Therefore, we can get that \begin{equation}
\label{fssn1}
\mff_{\s\s}^{(n-1)}/\gamma^{(n-1)}_{\s}=\sum_{\blam\in\P_n}\sum_{\substack{\u\in\Std(\blam)\\ \u\downarrow_{n-1}=\s}}(\gamma^{(n)}_\u/\gamma^{(n-1)}_{\s})\beta_{\u\u}^{\s\s}\mff_{\u\u}^{(n)}/\gamma^{(n)}_\u .
\end{equation}

Since $\mff_{\s\s}^{(n-1)}/\gamma^{(n-1)}_{\s}$ is a primitive idempotent and $\{\mff_{\u\u}^{(n)}/\gamma^{(n)}_\u|\u\in\Std(\blam),\blam\in\P_n\}$ is a complete set of pairwise orthogonal primitive idempotents in $\HH_{\ell,n}(q,\bQ)$, it follows that $(\gamma^{(n)}_\u/\gamma^{(n-1)}_{\s})\beta_{\u\u}^{\s\s}=1$ whenever $\u\in\Std(\blam)$, $\blam\in\P_n$, which satisfies $\beta_{\u\u}^{\s\s}\neq 0$ (and hence $\u\downarrow_{n-1}=\s$). Thus \begin{equation}\label{sscases}
\text{$\beta_{\u\u}^{\s\s}=\gamma^{(n-1)}_{\s}/\gamma^{(n)}_{\u}$ whenever $\beta_{\u\u}^{\s\s}\neq 0$.}
\end{equation}

On the other hand, we have that $\sum_{\bmu\in\P_{n-1}}\sum_{\s\in\Std(\bmu)}\mff_{\s\s}^{(n-1)}/\gamma^{(n-1)}_{\s}=1$. Combining this with (\ref{fssn1}), (\ref{sscases}) and the equality $\sum_{\blam\in\P_n}\sum_{\u\in\Std(\blam)}\mff_{\u\u}^{(n)}/\gamma^{(n)}_\u=1$ together we can deduce that $\beta_{\u\u}^{\s\s}\neq 0$ if and only if $u\downarrow_{n-1}=\s$. This completes the proof of part b) of the lemma.
\end{proof}

Let $\blam\in\P_n$ and $\alpha\in[\blam]$. If $[\blam]\setminus\{\alpha\}$ is again the Young diagram of a multipartition, then we say that $\alpha$ is a removable node of $[\blam]$.

\begin{lem}\label{gammann+1}
	Let $\bmu\in\P_{n-1}, \blam\in\P_n$ such that $\bmu=\blam\setminus\{\alpha\}$ for some removable node $\alpha$ of $[\blam]$. Let $\s,\t\in\Std(\bmu)$, $\u,\v\in\Std(\blam)$. If $\u\downarrow_{n-1}=\s$, then we have  	$$\frac{\gamma_{\s}^{(n-1)}}{\gamma_{\u}^{(n)}}=\frac{\gamma_{\t^{\bmu}}^{(n-1)}}{\gamma_{\a}^{(n)}},$$
where $\a\in\Std(\blam)$ is the unique standard $\blam$-tableau such that $\a\downarrow_{n-1}=\t^{\bmu}$.
\end{lem}
\begin{proof}
By assumption, $\u\downarrow_{n-1}=\s$. In particular, $\res_{\s}(k)=\res_{\u}(k)$, for any $1\leq k\leq n-1$. For any $\s\in\Std(\bmu)$, let $d(\s)\in\Sym_n$ such that $\t^{\bmu}d(\s)=\s$. We now fix a reduced expression $d(\s):=s_{i_{1}}s_{i_{2}}\cdots s_{i_{m}}$. We set $\s_{0}:=\t^{\bmu}$ and $\s_{k}:=\t^{\bmu}s_{i_{1}}s_{i_{2}}\cdots s_{i_{k}}$ for $1\leq k\leq m$. Then we can get a sequence of standard $\bmu$-tableaux
	$$\t^{\bmu}=\s_{0}\triangleright\s_{1}\triangleright\cdots\triangleright\s_{m}=\s.$$
For each $0\leq k\leq m$, we use $\u_k$ to denote the unique standard $\blam$-tableau such that $\u_k\downarrow_{n-1}=\s_k$. Then we have that $$
\a=\u_{0}\triangleright\u_{1}\triangleright\cdots\triangleright\u_{m}=\u .$$
In particular,
	$$\res_{\u_{k}}(j)=\res_{\s_{k}}(j),\quad \forall\,0\leq k\leq m,\, 1\leq j\leq n-1.$$

By the inductive definition of the $\gamma$-coefficients given in Definition \ref{gamma1}, we can deduce that for all $1\leq k\leq m$,  $$\frac{\gamma_{\s_{k}}^{(n-1)}}{\gamma_{\s_{k-1}}^{(n-1)}}=\frac{\gamma_{\u_{k}}^{(n)}}{\gamma_{\u_{k-1}}^{(n)}} .$$
It follows that
	$$
		\frac{\gamma_{\s}^{(n-1)}}{\gamma_{\t^{\bmu}}^{(n-1)}}
		=\frac{\gamma_{\s_{m}}^{(n-1)}}{\gamma_{\s_{m-1}}^{(n-1)}}\frac{\gamma_{\s_{m-1}}^{(n-1)}}{\gamma_{\s_{m-2}}^{(n-1)}}\cdots\frac{\gamma_{\s_{1}}^{(n-1)}}{\gamma_{\s_{0}}^{(n-1)}}
		=\frac{\gamma_{\u_{m}}^{(n)}}{\gamma_{\u_{m-1}}^{(n)}}\frac{\gamma_{\u_{m-1}}^{(n)}}{\gamma_{\u_{m-2}}^{(n)}}\cdots\frac{\gamma_{\u_{1}}^{(n)}}{\gamma_{\u_{0}}^{(n)}}=\frac{\gamma_{\u}^{(n)}}{\gamma_{\a}^{(n)}}.
	$$
	This completes the proof of the lemma.
\end{proof}

\medskip

\noindent
{\textbf{Proof of Theorem \ref{mainthm2}}}: Recall the invertible elements $\{\Phi_{\t}|\t\in\Std(\bmu)\}$ of $\HH_q(\Sym_{n-1})$ defined in Lemma \ref{M5prop}. By Lemma \ref{M5prop}, for $\s,\t\in\Std(\bmu)$, $\mff_{\s\t}^{(n-1)}=\Phi_{\s}^*\mff_{\t^\bmu\t^\bmu}^{(n-1)}\Phi_{\t}$. Applying Lemma \ref{fst3}, we get that
	$$\mff_{\t^{\bmu}\t^{\bmu}}^{(n-1)}=\sum_{\blam\in\P_n}\sum_{\substack{\a\in\Std(\blam)\\ \a\downarrow_{n-1}=\t^{\bmu}}}\frac{\gamma_{\t^{\bmu}}^{(n-1)}}{\gamma_{\a}^{(n)}}\mff_{\a\a}^{(n)}.$$

Let $\blam\in\P_{n}$. Note that $d(\s), d(\t)\in\Sym_{n-1}$. For any $\a\in \Std(\blam)$ satisfying $\a\downarrow_{n-1}=\t^{\bmu}$, it is clear that $\a d(\s), \a d(\t)\in \Std(\blam)$ and $\a d(\s)\downarrow_{n-1}=\s, \a d(\t)\downarrow_{n-1}=\t$. Therefore, it follows from the definitions of $\Phi_\s, \Phi_\t$ and Lemma \ref{seminormalB} that  $\Phi_\s^* \mff_{\a\a}^{(n)}\Phi_{\t}=\mff_{\a d(\s)\a d(\t)}^{(n)}$. Thus we have that
	$$\begin{aligned}
	\mff_{\s\t}^{(n-1)}&=\Phi_\s^*\mff_{\t^\bmu\t^\bmu}^{(n-1)}\Phi_\t\\
	&=\Phi_{\s}^*\bigg(\sum_{\blam\in\P_n}\sum_{\substack{\a\in\Std(\blam)\\ \a\downarrow_{n-1}=\t^{\bmu}}}\frac{\gamma_{\t^{\bmu}}^{(n-1)}}{\gamma_{\a}^{(n)}}\mff_{\a\a}^{(n)}\bigg)\Phi_{\t}\qquad\text{(By Lemma \ref{fst3})}\\
	&=\sum_{\blam\in\P_n}\sum_{\substack{\a\in\Std(\blam)\\ \a\downarrow_{n-1}=\t^{\bmu}}}\frac{\gamma_{\t^{\bmu}}^{(n-1)}}{\gamma_{\a}^{(n)}}\mff_{\a d(\s)\a d(\t)}^{(n)}\\
	&=\sum_{\blam\in\P_n}\sum_{\substack{\u,\v\in\Std(\blam)\\ \u\downarrow_{n-1}=\s\\ \v\downarrow_{n-1}=\t}}\frac{\gamma_{\t^{\bmu}}^{(n-1)}}{\gamma_{\a}^{(n)}}\mff_{\u\v}^{(n)},
	\end{aligned}$$
where the last equality follows because $\a d(\s)$ is the unique $\u\in\Std(\blam)$ satisfying $\u\downarrow_{n-1}=\s$, and $\a d(\t)$ is the unique $\v\in\Std(\blam)$ satisfying $\v\downarrow_{n-1}=\t$.	
Since $\frac{\gamma_{\t^{\bmu}}^{(n-1)}}{\gamma_{\a}^{(n)}}\in K^\times$, the above equality also implies that $\beta_{\u\v}^{\s\t}\neq 0$ if and only if $\u\downarrow_{n-1}=\s, \v\downarrow_{n-1}=\t$.

Finally, combining the above equality and Lemma \ref{gammann+1}, we can deduce that
	$$\beta_{\u\v}^{\s\t}=\frac{\gamma_{\t^{\bmu}}^{(n-1)}}{\gamma_{\a}^{(n)}}=\frac{\gamma_{\s}^{(n-1)}}{\gamma_{\u}^{(n)}}=\frac{\gamma_{\t}^{(n-1)}}{\gamma_{\v}^{(n)}}. $$
This completes the proof of the theorem.\qed
\medskip

%

\bigskip

\section{The degenerate case}

Let $\bu=(u_1,\cdots,u_\ell)$, where $u_1,\cdots,u_\ell \in K$. Let $H_{\ell,n}(\bu)$ be the degenerate cyclotomic Hecke algebra over $R$ with cyclotomic parameters $u_1,\cdots,u_\ell$. The purpose of this section is to give a proof of Theorem \ref{mainthm1b} and Theorem \ref{mainthm2b}. The argument of the proof is similar to the non-degenerate case. Throughout this section, we shall assume (\ref{ss0b}) holds. In particular, (\ref{separacontent1b}) holds and $H_{\ell,n}(\bu)$ is semisimple over $K$.


Let $\blam\in\P_n$. For any $\t=(\t^{(1)},\cdots,\t^{(\ell)})\in\Std(\blam)$ and any $1\leq k\leq n$, we define $$
\wres_{\t}(k)=j-i+u_c,\quad \text{if $k$ appears in row $i$ and column $j$ of $\t^{(c)}$}
$$
We also define $C(k):=\bigl\{\wres_{\t}(k)\bigm| \t\in\Std(\blam), \blam\in\P_n\bigr\}$.

\begin{dfn}(\cite[Definition 6.7]{AMR})\label{Ft2} Suppose (\ref{ss0b}) holds and $R=K$ is a field. Let $\blam\in\P_n$ and $\t\in\Std(\blam)$. We define
	$$\rmF_{\t}=\prod\limits^n\limits_{k=1}\prod\limits_{\substack{c\in C(k)\\c\neq \wres_{\t}(k)}}\frac{L_k-c}{\wres_{\t}(k)-c}.$$
\end{dfn}


\begin{dfn}{\rm (\cite[Lemma 6.10]{AMR})}\label{gamma2} Suppose (\ref{ss0b}) holds and $R=K$ is a field. Let $\blam\in\P_n$. The $r$-coefficients $\{\wgamma_\t|\t\in\Std(\blam),\blam\in\P_n\}$ are defined to be a multiset of invertible scalars in $K^\times$ which are uniquely determined by:
	\begin{enumerate}
		\item $\wgamma_{\t^{\blam}}=\Bigl(\prod\limits_{c=1}^\ell\prod\limits_{i\geq1}(\lambda^{(c)}_i)!\Bigr)\prod\limits_{1\leq s<t\leq\ell}\prod\limits_{1\leq j\leq \lambda_i^{(s)}}(j-i+u_s-u_t)$; and
		\item if $\s=\t(i,i+1)\triangleright\t$ then $$ \frac{\wgamma_{\t}}{\wgamma_{\s}}=\frac{(\wres_{\s}(i)-\wres_{\t}(i)+1)(\wres_{\s}(i)-\wres_{\t}(i)-1)}{(\wres_{\s}(i)-\wres_{\t}(i))^2}. $$
	\end{enumerate}
\end{dfn}


Let $\{\rmm_{\s\t}\ |\ \s,\t \in \Std(\blam), \blam\in\P_n\}$ be the cellular basis of $H_{\ell,n}(\bu)$ introduced in Section 2. Let $\{\rmf_{\s\t}\ |\ \s,\t \in \Std(\blam), \blam\in\P_n\}$ be the corresponding  seminormal basis of $H_{\ell,n}(\bu)$. For each $\lam\in\P_n$, we define $$
H_{\ell,n}^{\rhd\blam}:=\text{Span}_{R}\{\rmm_{\s\t}|\s,\t\in\Std(\bmu), \blam\lhd\bmu\in\P_n\},
$$
which is a cell ideal of $H_{\ell,n}(\bu)$ with respect to the cellular basis. For any $1\leq k\leq n$ and $\s,\t\in\Std(\blam)$, we have that \begin{equation}\label{L12}
\rmm_{\s\t}L_k=\wres_{\t}(k)\rmm_{\s\t}+\sum_{\substack{\v\in\Std(\blam)\\ \v\triangleright\t}} a_{\v}\rmm_{\s\v}\pmod{H_{\ell,n}^{\rhd\blam}},
 \end{equation}
where $a_\v\in K$ for each $\t\lhd\v\in\Std(\blam)$.

The dual seminormal basis of $H_{\ell,n}(\bu)$ can be constructed in the same manner as that of the non-degenerate cyclotomic Hecke algebras $\HH_{\ell,n}(q,\bQ)$. First, we recall the construction of the dual cellular basis for $H_{\ell,n}(\bu)$. Let $\blam\in\P_n$. We define  \begin{equation}
\rmn_{\t_\blam\t_\blam}:=(-1)^{\mathbf{n}(\blam)}\Bigl(\sum_{w\in\Sym_{\blam'}}(-1)^{\ell(w)}w\Bigr)\Bigl(\prod_{k=1}^{n}
\Bigl(\prod_{s=2}^{\ell}\prod_{k=1}^{|\lam^{(\ell)}|+|\lam^{(\ell-1)}|+\cdots+|\lam^{(\ell-s+2)}|}(L_k-u_{\ell-s+1})\Bigr),
\end{equation}
where $\mathbf{n}(\blam):=\sum_{i=1}^{\ell}(i-1)|\lam^{(i)}|$.

Recall that for any $\t\in\Std(\blam)$, $d'(\t)\in\Sym_n$ is such that $\t_\blam d'(\t)=\t$. For any  $\s,\t\in\Std(\blam)$, we define \begin{equation}\label{nst2}
\rmn_{\s\t}:=(-1)^{\ell(d'(\s))+\ell(d'(\t))}d'(\s)^{-1}\rmn_{\t_\blam\t_\blam}d'(\t). \end{equation}
Then, with respect to the poset $(\P_n,\unlhd)$ and the anti-involution ``$\ast$'', $\{\rmn_{\s\t}\ |\ \s,\t \in \Std(\blam), \blam\in\P_n\}$ forms another cellular basis of $H_{\ell,n}(\bu)$. We call it the {\bf dual cellular basis} of $H_{\ell,n}(\bu)$.

For each $\lam\in\P_n$, we define $$
\check{H}_{\ell,n}^{\lhd\blam}:=\text{$K$-Span}\{\rmn_{\s\t}|\s,\t\in\Std(\bmu), \blam\rhd\bmu\in\P_n\},
$$
which is a cell ideal of $H_{\ell,n}(\bu)$ related to the dual cellular basis. For any $1\leq k\leq n$ and $\s,\t\in\Std(\blam)$, we have that \begin{equation}\label{L22}
\rmn_{\s\t}L_k=\wres_{\t}(k)\rmn_{\s\t}+\sum_{\substack{\v\in\Std(\blam)\\ \v\lhd\t}} b_{\v}\rmn_{\s\v} \pmod{\check{H}_{\ell,n}^{\lhd\blam}}, \end{equation}
where $b_\v\in K$ for each $\t\rhd\v\in\Std(\blam)$.

\begin{dfn}\label{gdfn2} Suppose (\ref{ss0b}) holds and $R=K$ is a field. Let $\blam\in\P_n$. For any $\s,\t\in\Std(\blam)$, we define $$\rmg_{\s\t}:=\rmF_{\s}\rmn_{\s\t}\rmF_{\t}. $$
\end{dfn}

\begin{dfn}\text{(\cite[Definition 2.9]{Zh})}\label{prime2} Suppose that $\hat{u}_1,\cdots,\hat{u}_\ell$ are indeterminates over $\Z$. We set $\mathscr{A}_1:=\Z[\hat{u}_1,\cdots,\hat{u}_\ell]$ and
$\mathscr{K}_1:=\Q(\hat{u}_1,\cdots,\hat{u}_\ell)$. Let $H_{\ell,n}(\hat{\bu})$ be the degenerate cyclotomic Hecke algebra of type $G(\ell,1,n)$ over $\mathscr{A}_1$ with cyclotomic parameters $\hat{\bu}:=(\hat{u}_1,\cdots,\hat{u}_\ell)$. Set $H^{\mathscr{K}_1}_{\ell,n}(\hat{\bu}):=\mathscr{K}_1\otimes_{\mathscr{A}_1}H_{\ell,n}(\hat{\bu})$. It is clear that $H^{\mathscr{K}_1}_{\ell,n}(\hat{\bu})$ is semisimple.
In this case, we set $'$ to be the unique ring involution  of $H_{\ell,n}(\hat{\bu})$ which is defined on generators by $$
\hat{s}'_i:=-\hat{s}_i,\,\, {L}'_m:=-{L}_m,\,\,\hat{u}'_j:=-\hat{u}_{\ell-j+1},\quad 1\leq i<n,\, 1\leq m\leq n, 1\leq j\leq\ell .
$$
\end{dfn}

Clearly, $'$ naturally extends to a ring involution  of $H^{\mathscr{K}_1}_{\ell,n}(\hat{\bu})$. In particular, in this case $\fm'_{\s\t}=\fn_{\s'\t'}$, $(c_\t(k))'=-c_{\t'}(k)$ for any $1\leq k\leq n$. It follows from Definition \ref{Ft2} that \begin{equation}\label{2primes2}
\rmF'_\t=\rmF_{\t'},\quad \rmf'_{\s\t}=(\rmF_{\s}\rmm_{\s\t}\rmF_{\t})'=\rmF'_{\s}\rmm'_{\s\t}\rmF'_{\t}=\rmF_{\s'}\rmn_{\s'\t'}\rmF_{\t'}=\rmg_{\s'\t'}.
\end{equation}
For any rational function $f$ on $\hat{u}_1,\cdots,\hat{u}_\ell$, $f'$ is the rational function obtained from $f$ by substituting $\hat{u}_i$ with $-\hat{u}_{\ell-i+1}$ for each $1\leq i\leq\ell$.

By Definition \ref{gamma1}, for each $\t\in\Std(\blam)$, $\wgamma_\t$ is given by the evaluation of a rational function $\wgamma_\t(\hat{u}_1,\cdots,\hat{u}_\ell)$ at $\hat{u}_i:=u_i$, $1\leq i\leq\ell$. Thus the notation $\wgamma'_\t:=1_K\otimes_{\mathscr{A}_1}{\wgamma}'_{\t}(\hat{u}_1,\cdots,\hat{u}_\ell)$ makes sense.

\begin{rem}\label{convention2} Note that our notations $\rmn_{\s\t}, \rmg_{\s\t}$ differ with the corresponding notations in \cite{Zh} by a conjugation. Namely, the readers should identify the elements $\rmn_{\s\t}, \rmg_{\s\t}$ in the current paper with the elements $n_{\s'\t'}, g_{\s'\t'}$ in \cite{Zh}. In particular, our dual cellular basis $\{\rmn_{\s\t}\}$ use the partial order $\unlhd$, while \cite{Zh} use the opposite partial order $\unrhd$ for the dual cellular basis.
\end{rem}

The following corollary and lemma can be proved in a similar way as in the non-degenerate case.

\begin{cor}\label{obobn22}
	Suppose (\ref{ss0b}) holds and $R=K$ is a field. Then \begin{equation}\label{dualsem2}\{\rmg_{\s\t}\ |\ \s,\t \in \Std(\blam), \blam\in\P_n\}\end{equation} is a basis of $H_{\ell,n}(\bu)$. Moreover,
\begin{enumerate}
\item[1)] if $\s, \t, \u$ and $\v$ are standard tableaux, then $\rmg_{\s\t}\rmg_{\u\v}=\delta_{\t\u}\wgamma'_{\t'}\rmg_{\s\v}$;
\item[2)] if $\blam\in\P_n$, $\s,\t\in\Std(\blam)$ and $1\leq k\leq n$, then $\rmg_{\s\t}L_k=c_\t(k)\rmg_{\s\t}$, $L_k\rmg_{\s\t}=c_\s(k)\rmg_{\s\t}$;
\item[3)] for each $\lam\in\P_n$ and $\t\in\Std(\blam)$, $\rmF_{\t}=\rmg_{\t\t}/r'_{\t'}$;
\end{enumerate}
\end{cor}
We call (\ref{dualsem2}) the {\bf dual seminormal basis} of $H_{\ell,n}(\bu)$ corresponding to the dual cellular basis $\{\rmn_{\s\t}\ |\ \s,\t \in \Std(\blam), \blam\in\P_n\}$.

Note that in general we have $\wgamma'_\t\neq\wgamma_{\t'}$. For example, if $\ell=1, u_1=0$, $\lam=(2,1)$, $\t=\t^\lam s_2$, then $$
\wgamma_\t=\frac{3\cdot 1\cdot 2}{2^2},\quad \wgamma'_\t=\frac{3\cdot 1\cdot 2}{2^2}\neq\wgamma_{\t'}=2 .
$$

\begin{lem}\label{cor22} Suppose (\ref{ss0b}) holds and $R=K$ is a field. Suppose that $\blam$ is a multipartition of $n$ and $\s,\t\in \Std(\blam)$.
	\begin{enumerate}
		\item[1)] For any standard tableau $\t$, we have
		$$\rmg_{\t'\t'}=\rmf'_{\t\t}=\wgamma_{\t}'\rmF_{\t'}=\frac{\wgamma_{\t}'}{\wgamma_{\t'}}\rmf_{\t'\t'}.$$
		\item[2)] There exists $a_{\s\t}\in K^\times$ such that $\rmg_{\s\t}=a_{\s\t}\rmf_{\s\t}$. Moreover, $a_{\s\t}^2=r_{\s'}'r_{\t'}'/\wgamma_{\s}\wgamma_{\t}$.
	\end{enumerate}
\end{lem}

For the reader's convenience, we include below a lemma which gives a recursive formula for the $r'$-coefficients associated to the dual seminormal bases.

\begin{lem}
	Suppose (\ref{ss0b}) holds. Let $\blam\in\P_n$. We define a multiset of elements $\{r^{(n)}_\t\in K^\times|\t\in\Std(\blam),\blam\in\P_n\}$ in $K^\times$ as follows:
	\begin{enumerate}
		\item $r'_{(\t_\blam)'}=r'_{\t^{\blam'}}=(-1)^{C}\Bigl(\prod_{l=1}^{\ell}\prod_{i\geq 1}(\lam^{(l)'})_i!\Bigr)\prod\limits_{1\leq t<s\leq\ell}\prod\limits_{1\leq j\leq\lam_i^{(s)}}(j-i+u_s-u_t)$, where $$
C=\sum\limits_{1<s\leq\ell}\sum\limits_{{i\geq 1}}(s-1)\lam_i^{(s)}; $$ and
		\item if $\s=\t(i,i+1)\triangleleft\t$ then $$ \frac{r'_{\t'}}{r'_{\s'}}=\frac{(1+c_{\t}(i)-c_{\s}(i))(c_{\t}(i)-c_{\s}(i)-1)}{(c_{\t}(i)-c_{\s}(i))^2}=\frac{r_{\s}}{r_{\t}}. $$
	\end{enumerate}
\end{lem}
\begin{proof} This follows from Definition \ref{prime2} and the equality $(c_{\t}(k))'=c_{\t'}(k)$.
\end{proof}

\begin{lem}{\rm (\cite[Lemma 3.8]{HuMathas:SeminormalQuiver})}\label{seminorma2B} Suppose (\ref{ss0b}) holds and $R=K$ is a field. Let $\blam\in\P_n$ and $\s,\u\in\Std(\blam)$. Let $i,m$ be integers with $1\leq i< n, 1\leq m\leq n$ and $\t:=\s(i,i+1)$. If $\t$ is standard then
	$$\begin{aligned}\rmf_{\u\s}s_i&=\begin{cases*}
		a_i(\s)\rmf_{\u\s}+\rmf_{\u\t}, &\text{if $\t\triangleleft\s$},\\
		a_i(\s)\rmf_{\u\s}+b_i(\s)\rmf_{\u\t}, &\text{if $\s\triangleleft\t$},
	\end{cases*}\\
\rmf_{\u\s}L_m&=c_\s(m)\rmf_{\u\s},
\end{aligned}
$$
where $$
a_i(\s)=\frac{1}{c_\s(i+1)-c_\s(i)},\quad b_i(\s):={\wgamma_\s}/{\wgamma_\t}=\frac{(c_{\s}(i)-c_{\s}(i+1)+1)(c_{\s}(i)-c_{\s}(i+1)-1)}{(c_{\s}(i+1)-c_{\s}(i))^2}.
$$
If $\t$ is not standard then
	$$\rmf_{\u\s}s_i=\begin{cases*}
		\rmf_{\u\s}, &\text{if $i$ and $i+1$ are in the same row of $\s$},\\
		-\rmf_{\u\s}, &\text{if $i$ and $i+1$ are in the same column of $\s$}.
	\end{cases*}$$
\end{lem}

The following lemma can be proved in the same way as the proof of \cite[Proposition 4.1, Lemma 4.3]{Ma}.

\begin{lem}\label{M5prop2} Suppose (\ref{ss0b}) holds and $R=K$ is a field.
	Let $\blam\in\P_n$ and $i$ an integer with $1\leq i<n$. Then there exist a family of invertible elements $\{\phi_{\t}|\t\in\Std(\blam)\}$ in $K[\Sym_n]$ such that
	\begin{enumerate}
		\item[(i)] for any $\s,\t\in\Std(\blam)$, $\rmf_{\s\t}=\phi_\s^*\rmf_{\t^\blam\t^\blam}\phi_\t$;
		\item[(ii)] $\phi_{\t^\blam}=1$, and if $\s:=\t(i,i+1)\lhd\t$, then $\phi_\s=\phi_\t(s_i-a_i(\t))$.
	\end{enumerate}
\end{lem}

\begin{lem}\label{phiprime2} Suppose (\ref{ss0b}) holds and $R=K$ is a field. Let $\blam\in\P_n$ and $\t\in\Std(\blam)$. Let $i$ be an integer with $1\leq i<n$. Suppose that $u_1,\cdots,u_\ell$ are indeterminates over $\Z$.
If $\s:=\t(i,i+1)\in\Std(\blam)$ with $\s\lhd\t$, then $$
\phi_\s'=-\phi_\t'(s_i-a_i(\t')). $$
\end{lem}

\begin{proof} To prove the lemma, we can assume without loss of generality that $u_1,\cdots,u_\ell$ are indeterminates over $\Z$. In this case, we can use the ring involution $'$ introduced in \cite[\S3]{Ma} which is defined on generators by $$
s'_i:=-s_i,\,\, L'_m:=-L'_m,\,\,u'_j:=-u_{\ell-j+1},\quad \forall\,1\leq i<n,\, 1\leq m\leq n,\,1\leq j\leq\ell .
$$
By definition, $\phi'_\s=\phi'_\t(-s_i-a_i(\t)')$. Thus it suffices to show that $a_i(\t)'=-a_i(\t')$.

By definition, we have that $(c_{\t}(k))'=-c_{\t'}(k)$. It follows that
	$$a_i(\t)'=\big(\frac{1}{c_\t(i+1)-c_\t(i)}\big)'
			=-\frac{1}{c_{\t'}(i)-c_{\t'}(i+1)}.$$
	 for $1\leq k\leq n$. Hence, we can get that $a_i(\t)'=\frac{1}{c_{\t'}(i+1)-c_{\t'}(i)}=-a_i(\t')$.
This completes the proof of the lemma.
\end{proof}


The following lemma can be proved by using Lemma \ref{phiprime2} and a similar argument used in the proof of Lemma \ref{phiprime}.

\begin{lem}\label{scalaralt2} Suppose (\ref{ss0b}) holds and $R=K$ is a field. Let $\blam\in\P_n$ and $\t\in\Std(\blam)$. Suppose that $u_1,\cdots,u_\ell$ are indeterminates over $\Z$. Then we have $$
\rmf_{\t_{\blam'}\t_{\blam'}}\phi'_{\t}=(-1)^{\ell(d(\t))}\frac{\wgamma_{\t_{\blam'}}}{\wgamma_{\t'}}\rmf_{\t_{\blam'}\t'},\quad
(\phi_{\t}^*)'\rmf_{\t_{\blam'}\t_{\blam'}}=(-1)^{\ell(d(\t))}\frac{\wgamma_{\t_{\blam'}}}{\wgamma_{\t'}}\rmf_{\t'\t_{\blam'}} .
$$
\end{lem}

Let $\bmu\in\P_{n-1}$ and $\s,\t\in\Std(\bmu)$. By (\ref{betab}), we have $$
{\rm f}_{\s\t}^{(n-1)}=\sum_{\blam\in\P_n}\sum_{\u,\v\in\Std(\blam)}b_{\u\v}^{\s\t}{\rm f}_{\u\v}^{(n)},
$$
where $b_{\u\v}^{\s\t}\in K$ for each pair $(\u,\v)$.

Replacing the $\gamma$-coefficients and the element $\Phi_\t$ of $\HH_{\ell,n}(q,\bQ)$ with the $r$-coefficients and the element $\phi_\t$ of $H_{\ell,n}(\bu)$, the following lemmas can be proved in the same way as the proof of
Lemmas \ref{gam'tgamt'}, \ref{fst3}, \ref{gammann+1}.

\begin{lem}\label{gam'tgamt'2} Suppose (\ref{ss0b}) holds and $R=K$ is a field.
	Let $\blam\in\P_n$ be a multipartition of $n$ and $\t\in\Std(\blam)$. Then we have that
	$$\wgamma_{\t'}\wgamma_{\t}'=\wgamma_{\t_{\blam'}}\wgamma_{\t^{\blam}}'.$$
\end{lem}

\begin{lem}\label{fst32} Suppose (\ref{ss0b}) holds and $R=K$ is a field. Let $\bmu\in\P_{n-1}, \blam\in\P_n$, and $\s,\t\in\Std(\bmu)$, $\u,\v\in\Std(\blam)$. Then \begin{enumerate}
\item $b_{\u\v}^{\s\t}\neq 0$ only if $\u\downarrow_{n-1}=\s$ and $\v\downarrow_{n-1}=\t$;
\item $b_{\u\u}^{\s\s}\neq 0$ if and only if $\u\downarrow_{n-1}=\s$. In that case, $b_{\u\u}^{\s\s}=\wgamma^{(n-1)}_{\s}/\wgamma^{(n)}_{\u}$.
\end{enumerate}
\end{lem}

\begin{lem}\label{gammann+12}
	Let $\bmu\in\P_{n-1}, \blam\in\P_n$ such that $\bmu=\blam\setminus\{\gamma\}$ for some removable node $\gamma$ of $[\blam]$. Let $\s,\t\in\Std(\bmu)$, $\u,\v\in\Std(\blam)$. If $\u\downarrow_{n-1}=\s$, then we have  	$$\frac{\wgamma_{\s}^{(n-1)}}{\wgamma_{\u}^{(n)}}=\frac{\wgamma_{\t^{\bmu}}^{(n-1)}}{\wgamma_{\a}^{(n)}},$$
where $\a\in\Std(\blam)$ is the unique standard $\blam$-tableau such that $\a\downarrow_{n-1}=\t^{\bmu}$.
\end{lem}

\medskip
\noindent
{\textbf{Proof of Theorems \ref{mainthm1b}, \ref{mainthm2b}}}: Theorem \ref{mainthm1b} follows from Lemmas \ref{scalaralt2}, \ref{scalaralt2} and a similar argument used in the proof of Theorem \ref{mainthm1}.
Theorem \ref{mainthm2b} follows from Lemmas \ref{fst32}, \ref{gammann+12} and a similar argument used in the proof of Theorem \ref{mainthm2}.\qed


\bigskip
\bigskip
\bigskip
\begin{thebibliography}{2}

\bibitem{A1} {\sc S.~Ariki}, {\em On the semi-simplicity of the Hecke algebra of $(\Z/r\Z)\wr\Sym_n$}, J. Algebra, {\bf 169} (1994), 216--225.

\bibitem{A2} \leavevmode\vrule height 2pt depth -1.6pt width 23pt, {\em Representation theory of a Hecke algebra of $G(r,p,n)$}, J. Algebra, {\bf 177} (1995), 164--185.

\bibitem{A3} \leavevmode\vrule height 2pt depth -1.6pt width 23pt, {\em On the decomposition numbers of the Hecke algebra of $G(m,1,n)$}, J. Math.
Kyoto Univ., {\bf 36} (1996), 789--808.

\bibitem{AK} {\sc S.~Ariki and K.~Koike}, {\em  A Hecke algebra of $(\Z/r\Z)\wr\Sym_n$ and construction of its representations}, Adv. Math., {\bf 106} (1994), 216--243.

\bibitem{AMR} {\sc S.~Ariki, A.~Mathas and H.~Rui}, {\em Cyclotomic Nazarov-Wenzl algebras}, Nagoya Math. J., {\bf 182} (2006), 47--134 (special
issue in honour of George Lusztig).


\bibitem{BK}
{\sc J.~Brundan and A.~Kleshchev}, {\em Schur-Weyl duality for higher levels}, Selecta Math., {\bf 14} (2008), 1--57.

\bibitem{BK:GradedKL}
\leavevmode\vrule height 2pt depth -1.6pt width 23pt,  {\em Blocks of cyclotomic {H}ecke algebras and {K}hovanov-{L}auda algebras}, Invent. Math., {\bf 178}
  (2009), 451--484.

\bibitem{BK:Gradedecomp}
\leavevmode\vrule height 2pt depth -1.6pt width 23pt,  {\em Graded decomposition numbers for cyclotomicHecke algebras}, Adv. Math., {\bf 222} 1883--1942 (2009)





\bibitem{BM:cyc}
{\sc M.~Brou{\'e} and G.~Malle}, {\em Zyklotomische {H}eckealgebren},
  Ast\'erisque, \textbf{212} (1993), 119--189.
\newblock Repr{\'e}sentations unipotentes g{\'e}n{\'e}riques et blocs des
  groupes r{\'e}ductifs finis.


\bibitem{Ch} {\sc I.~Cherednik}, {\em A new interpretation of Gelfand-Tzetlin bases}, Duke Math. J., {\bf 54} (1987), 563--577.

\bibitem{DJM} {\sc R.~Dipper, G.D.~James and E.~Murphy}, {\em Hecke algebras of type
  $B_n$ at roots of unity}, Proc. London Math. Soc., {\bf 70}(3) (1995),
  505--528.

\bibitem{DJM:cyc}
{\sc R.~Dipper, G.~James, and A.~Mathas}, {{\em Cyclotomic {$q$}-{S}chur algebras}}, Math. Z., {\bf 229} (1998), 385--416.



\bibitem{GL} {\sc J.J.~Graham and G.I.~Lehrer}, {\em Cellular algebras}, Invent. Math., {\bf 123} (1996), 1--34.

\bibitem{Gro}
{\sc I.~Grojnowski}, {\em Affine $\widehat{\mathfrak{sl}}_p$ controls the modular representation theory of the symmetric group and related Hecke algebras}, preprint, arXiv:math.RT/9907129, 1999.


\bibitem{HuMathas:GradedCellular}
{\sc J.~Hu and A.~Mathas},
{\em Graded cellular
  bases for the cyclotomic {K}hovanov-{L}auda-{R}ouquier algebras of type
  {$A$}}, Adv. Math., \textbf{225} (2010), 598--642.

\bibitem{HuMathas:SeminormalQuiver}
\leavevmode\vrule height 2pt depth -1.6pt width 23pt,
{\em Seminormal forms
  and cyclotomic quiver {H}ecke algebras of type {A}}, Math. Ann.,
  \textbf{364} (2016), 1189--1254.

\bibitem{HuMathas:TAMS}
\leavevmode\vrule height 2pt depth -1.6pt width 23pt,
{\em Fayers' conjecture and the socle of cyclotomic Weyl modules}, Trans. Amer. Math. Soc., {\bf 371}(2) (2019), 1271--1307.




\bibitem{Klesh:book}
{\sc A.~S. Kleshchev}, {\em Linear and projective representations of symmetric
  groups}, CUP, 2005.


\bibitem{JM} {\sc G.D.~James and A.~Mathas}, {\em The Jantzen sum formula for cyclotomic $q$-Schur algebras}, Trans. Amer. Math.
Soc., {\bf 352} (2000), 5381--5404.


\bibitem{Ma1} {\sc A.~Mathas}, {\em Iwahori-Hecke algebras and Schur algebras of the symmetric group}, Amer. Math. Soc., 1999.

\bibitem{Ma} \leavevmode\vrule height 2pt depth -1.6pt width 23pt,  {\em Matrix units and generic degrees for the Ariki-Koike algebras}, {J. Algebra,} {\bf 281} (2004), 695--730.

\bibitem{Ma2} \leavevmode\vrule height 2pt depth -1.6pt width 23pt,  {\em Seminormal forms and Gram determinants for cellular algebras}, J. Reine Angew. Math.,
{\bf 619} (2008), 141-–173 (With an appendix by Marcos Soriano).

\bibitem{Mur1}
{\sc G.E.~Murphy}, {\em A new construction of Young’s semi-normal representation of the symmetric groups}, J. Algebra, {\bf 69} (1981), 287--297.

\bibitem{Mur2}
\leavevmode\vrule height 2pt depth -1.6pt width 23pt, {\em The idempotents of the symmetric group and Nakayama's conjecture}, J. Algebra, {\bf 81} (1983), 258--265.

\bibitem{Mur3}
\leavevmode\vrule height 2pt depth -1.6pt width 23pt,  {\em On the representation theory of the symmetric groups and associated Hecke algebras}, J. Algebra, {\bf 152} (1992), 492--513.

\bibitem{Mur4}
\leavevmode\vrule height 2pt depth -1.6pt width 23pt,  {\em The representations of Hecke algebras of type $A_n$}, J. Algebra, {\bf 173} (1995), 97--121.

\bibitem{Zh} {\sc D.~Zhao}, {\em Schur elements of the degenerate cyclotomic Hecke algebras}, Israel Journal of Mathematics, {\bf 205}, (2015), 485--507.

\end{thebibliography}
\end{document}